\documentclass[]{amsart}
\usepackage{amsmath}
\usepackage[english]{babel}
\usepackage[utf8]{inputenc}
\usepackage{amsfonts}
\usepackage{amsthm}
\usepackage{graphicx}
\usepackage{color}
\newtheorem{thm}{Theorem}[section]

\newtheorem{cor}[thm]{Corollary}
\newtheorem{lem}[thm]{Lemma}
\theoremstyle{remark}
\newtheorem{rmk}[thm]{Remark}
\theoremstyle{definition}

\DeclareMathOperator{\E}{\mathbb{E}}
\DeclareMathOperator{\N}{\mathbb{N}}
\DeclareMathOperator{\R}{\mathbb{R}}
\DeclareMathOperator{\C}{\mathbb{C}}

\DeclareMathOperator{\cP}{\mathcal{P}}
\DeclareMathOperator{\cA}{\mathcal{A}}
\DeclareMathOperator{\cN}{\mathcal{N}}
\DeclareMathOperator{\cL}{\mathcal{L}}
\DeclareMathOperator{\wcP}{\widetilde{\mathcal{P}}}
\DeclareMathOperator{\wcL}{\widetilde{\mathcal{L}}}
\DeclareMathOperator{\cI}{\mathcal{I}}

\DeclareMathOperator{\cS}{\mathcal{S}}

\DeclareMathOperator{\bP}{\mathbb{P}}

\newcommand{\pd}[2]{\frac{\partial #1}{\partial #2}}

\newcommand{\der}[2]{\frac{d #1}{d #2}}

\title{Fractional Erlang Queues}

\author{Giacomo Ascione$^\ast$}
\address{$^\ast$ Dipartimento di Matematica e Applicazioni ``Renato Caccioppoli'', Università degli Studi di Napoli Federico II, 80126 Napoli, Italy}
\author{Nikolai Leonenko$^\dagger$}
\thanks{N. Leonenko was supported in particular by Australian Research Council's Discovery Projects funding scheme (project DP160101366), and  by project MTM2015-71839-P of MINECO, Spain (co-funded with FEDER funds).}
\address{$^\dagger$ School of Mathematics, Cardiff University, Cardiff CF24 4AG, UK}
\author{Enrica Pirozzi$^\ast$}
\email{giacomo.ascione@unina.it \\
 leonenkon@cardiff.ac.uk \\
enrica.pirozzi@unina.it}

\begin{document}
\maketitle

\begin{abstract}
We introduce a fractional generalization of the Erlang Queues $M/E_k/1$. Such process is obtained through a time-change via inverse stable subordinator of the classical queue process. We first exploit the (fractional) Kolmogorov forward equation for such process, then we use such equation to obtain an interpretation of this process in the queuing theory context. Then we also exploit the transient state probabilities and some features of this fractional queue model, such as the mean queue length, the distribution of the busy periods and some conditional distributions of the waiting times. Finally, we provide some algorithms to simulate their sample paths.
\end{abstract}

\section{Introduction}
The Erlang Queue is one of the most popular model for a queue process. It is characterized by arrivals determined by a Poisson process and independent Erlang service times.\\
We introduce a non-Markovian generalization of the classical Erlang $M/E_k/1$ queue process based on the fractional Poisson process, FPP for short (\cite{AlettiLeonenkoMarzbach2018}). There are essentially three approaches to the concept of FPP. The ``renewal'' approach consists of considering the characterization of the Poisson process as a sum of independent non-negative random variables and, instea of the assumption that these random variables have an exponential distribution, we assume that they have the Mittag-Leffler distribution (\cite{Laskin2003,MainardiGorenfloScalas2007,MainardiGorenfloVivoli2005}). In \cite{BeghinOrsingher2009,BeghinOrsingher2010}, the renewal approach to the FPP is developed and it is proven that its $n$-dimensional distributions coincide with the solution to ``fractionalized'' differential-difference equations with fractional derivatives in time. Finally, using ``inverse subordination'', a FPP can be constructed as a time-changed classical Poisson process (\cite{MeerschaertNaneVellaisamy2011}).\\
The first fractional generalization of the classical $M/M/1$ queue process was proposed in \cite{CahoyPolitoPhoha2015}. The transient behaviour of the fractional $M/M/1$ queues with catastrophes (which in the classical case are studied for instance in \cite{DiCrescenzo2003} and further generalized in \cite{DiCrescenzo2018,DiCrescenzo2018b,GiornoNobilePirozzi2018}) is investigated in \cite{AscioneLeonenkoPirozzi2018}.\\
In this paper, we present a study of the transient behaviour of the fractional Erlang queue $M/E_k/1$.\\
In Section \ref{sec2}, we recall the basic definitions and properties for a classical Erlang queue $M/E_k/1$. In particular we recall the formulas for the transient state probabilities, the mean queue length and the distribution of the busy period.\\
In Section \ref{sec3} we give the definition of Fractional Erlang Queue and we determine the fractional forward Kolmogorov equation for the transient state probabilities of such queue.\\
In Section \ref{sec4} the Fractional Erlang Queue process is investigated in terms of inter-arrival times and service times. Such investigation leads to an interpretation of this process in a queueing theory context and the characterization of the distributions of the inter-arrival, inter-phase and service times.\\
In Section \ref{sec5} formulas for transient state probabilities of such process and their Laplace transform are finally obtained.
In Section \ref{sec6} some characteristics of the process are studied. In particular in Subsection \ref{subsect61}, a fractional differential equation for the mean queue length is provided. Moreover, we find also a formula for the mean queue length and its Laplace transform and, by using such formula, we provide an expression for the fractional integral of the probability that the queue is empty. In Subsection \ref{subsec62} we provide also a formula for the distribution of the busy period. In Subsection \ref{subsec63} we focus on the waiting times, determining a formula for a particular conditional waiting time.\\
Finally, in Section \ref{sec7}, we provide an algorithm for sample paths simulation of such process, which is a modified version of the well-known Gillespie algorithm (see \cite{Gillespie1977}).
\section{Erlang Queues}\label{sec2}
Let us recall the definition of Erlang queue. A $M/E_k/1$ queue is a queue with Poisson input (and so exponential independent interarrival times) with parameter $\lambda$ and an Erlang service system with shape parameter $k$ and rate $\mu$. Such service system can be seen as a system composed of $k$ phases each with exponential service time with parameter $k\mu$.\\
Let us denote with $N(t)$ the number of the customers in the system at the time $t \ge 0$. Moreover, if the queue is not empty, let us denote with $S(t)$ the phase of the customer that is currently being served at the time $t \ge 0$. If $N(t)=0$ let us denote also $S(t)=0$. Finally denote $Q(t)=(N(t),S(t))$ for $t \ge 0$, that is a sort of \textit{state-phase} process with state space:
\begin{equation*}
\cS=\{(n,s)\in \N \times \N: \ n \ge 1, \ 1 \le s \le k\}\cup\{(0,0)\}
\end{equation*}
and let us denote $\cS^*=\{(n,s)\in \N \times \N: \ n \ge 1, \ 1 \le s \le k\}$.
Let us define the transient state-phase probability functions as 
\begin{align*}
p_{n,s}(t)&=\bP(Q(t)=(n,s)|Q(0)=(0,0)) & (n,s)\in \cS^*, \ t \ge 0 \\
p_0(t)&=\bP(Q(t)=(0,0)|Q(0)=(0,0)), \ t \ge 0.
\end{align*}
It is well known (see for instance \cite{Saaty1961}) that the function $p_{n,s}(t)$ and $p_0(t)$ satisfy the following difference-differential equations:
{\footnotesize \begin{equation}\label{eq:syssp}
\begin{cases}
\der{p_0}{t}(t)=-\lambda p_0(t)+k\mu p_{1,1}(t) \\
\der{p_{1,s}}{t}(t)=-(\lambda+k\mu)p_{1,s}(t)+k\mu p_{1,s+1}(t) & 1 \le s \le k-1 \\
\der{p_{1,k}}{t}(t)=-(\lambda+k\mu)p_{1,k}(t)+k\mu p_{2,1}(t)+\lambda p_0(t) \\
\der{p_{n,s}}{t}(t)=-(\lambda+k\mu)p_{n,s}(t)+k\mu p_{n,s+1}(t)+\lambda p_{n-1,s}(t) & n\ge 2, \ 1 \le s \le k-1 \\
\der{p_{n,k}}{t}(t)=-(\lambda+k\mu)p_{n,k}(t)+k\mu p_{n+1,1}(t) + \lambda p_{n-1,k}(t) & n\ge 2\\
p_0(0)=1 \\
p_{n,s}(0)=0 & n \ge 1, \ 1 \le s \le k.
\end{cases}
\end{equation}}
In \cite{GrifLeoWill2005} a solution of this system is obtained by means of generalized modified Bessel functions defined in \cite{Luchak1958} as
\begin{equation}\label{eq:genmodBes1}
I_n^k(t)=\left(\frac{t}{2}\right)^n\sum_{r=0}^{+\infty}\frac{\left(\frac{t}{2}\right)^{r(k+1)}}{r!\Gamma(n+rk+1)}, \quad t \in \C, \ n=0,1,2,\dots, \ k=1,2,\dots
\end{equation}
and in \cite{GrifLeoWill2005b} as
\begin{multline}\label{eq:genmodBes2}
I_n^{k,s}(t)=\left(\frac{t}{2}\right)^{n+k-s}\sum_{r=0}^{+\infty}\frac{\left(\frac{t}{2}\right)^{r(k+1)}}{(k(r+1)-s)!\Gamma(n+r+1)} \\ t \in \C, \ n=0,1,2,\dots, \ k=1,2,\dots, \ s=1,2,\dots,k.
\end{multline}
In particular we have:
\begin{equation}
p_0(t)=\sum_{m=1}^{+\infty}m\left(\frac{\lambda}{k\mu}\right)^{-\frac{m}{k+1}}\frac{I_m^k\left(2\left(\frac{\lambda}{k\mu}\right)^{\frac{1}{k+1}}k\mu t\right)}{k\mu t}e^{-(\lambda+k\mu)t}
\end{equation}
while for $n\ge 1$ and $1 \le s \le k-1$:
{\small \begin{align*}
\begin{split}
p_{n,s}(t)&=\left(\frac{\lambda}{k\mu}\right)^{\frac{k(n-1)+s}{k+1}}I_n^{k,s}(2(\lambda(k\mu)^k)^{\frac{1}{k+1}}t)e^{-(\lambda+k\mu)t}\\&+k\mu\left(\frac{\lambda}{k\mu}\right)^{\frac{k(n-1)+s}{k+1}}\int_0^tp_0(z)I_n^{k,s}(2(\lambda(k\mu)^k)^{\frac{1}{k+1}}(t-z))e^{-(\lambda+k\mu)(t-z)}dz\\&-k\mu\left(\frac{\lambda}{k\mu}\right)^{\frac{k(n-1)+s+1}{k+1}}\int_0^tp_0(z)I_n^{k,s+1}(2(\lambda(k\mu)^k)^{\frac{1}{k+1}}(t-z))e^{-(\lambda+k\mu)(t-z)}dz
\end{split}
\end{align*}}
and finally for $n \ge 1$ and $s=k$:
\begin{align*}
\begin{split}
p_{n,k}(t)&=\left(\frac{\lambda}{k\mu}\right)^{\frac{kn}{k+1}}I_n^{k,k}
(2(\lambda(k\mu)^k)^{\frac{1}{k+1}}t)e^{-(\lambda+k\mu)t}\\
&+k\mu\left(\frac{\lambda}{k\mu}\right)^{\frac{kn}{k+1}}
\int_0^tp_0(z)I_n^{k,k}(2(\lambda(k\mu)^k)^{\frac{1}{k+1}}(t-z))
e^{-(\lambda+k\mu)(t-z)}dz\\
&-k\mu\left(\frac{\lambda}{k\mu}\right)^{\frac{kn+1}{k+1}}\int_0^tp_0(z)
I_{n+1}^{k,1}(2(\lambda(k\mu)^k)^{\frac{1}{k+1}}(t-z))e^{-(\lambda+k\mu)(t-z)}dz.
\end{split}
\end{align*}
It will be useful in the following to write the Bessel functions explicitly as power series. In this way we have
\begin{equation}\label{eq:p01}
p_0(t)=\sum_{m=1}^{+\infty}\sum_{r=0}^{+\infty}\frac{m\lambda^r(k\mu)^{m+rk-1}}{r!\Gamma(n+rk+1)}t^{m+r(k+1)-1}e^{-(\lambda+k\mu)t}
\end{equation}
while for $n \ge 1$ and $1 \le s \le k-1$ we have
\begin{align}\label{eq:pns1}
\begin{split}
p_{n,s}(t)&=\sum_{j=0}^{+\infty}\frac{\lambda^{n+j}(k\mu)^{k(j+1)-s}}{(k(j+1)-s)!\Gamma(n+j+1)}t^{n+k-s+j(k+1)}e^{-(\lambda+k\mu)t}\\&+\sum_{j=0}^{+\infty}\frac{\lambda^{n+j}(k\mu)^{k(j+1)-s+1}}{(k(j+1)-s)!\Gamma(n+j+1)}\\&\times\int_0^{t}p_0(z)(t-z)^{n+k-s+j(k+1)}e^{-(\lambda+k\mu)(t-z)}dz\\
&-\sum_{j=0}^{+\infty}\frac{\lambda^{n+j}(k\mu)^{k(j+1)-s}}{(k(j+1)-s-1)!\Gamma(n+j+1)}\\&\times\int_0^tp_0(z)(t-z)^{n+k-s-1+j(k+1)}e^{-(\lambda+k\mu)(t-z)}dz
\end{split}
\end{align}
and finally for $n\ge 1$ and $s=k$ we have
\begin{align}\label{eq:pnk1}
\begin{split}
p_{n,k}(t)&=\sum_{j=0}^{+\infty}\frac{\lambda^{n+j}(k\mu)^{kj}}{(kj)!\Gamma(n+j+1)}t^{n+j(k+1)}e^{-(\lambda+k\mu)t}+\\
&\hspace*{-1cm}+\sum_{j=0}^{+\infty}\frac{\lambda^{n+j}(k\mu)^{kj+1}}{(kj)!\Gamma(n+j+1)}\times\int_0^tp_0(z)(t-z)^{n+j(k+1)}e^{-(\lambda+k\mu)(t-z)}dz\\
&\hspace*{-1cm}-\sum_{j=0}^{+\infty}\frac{\lambda^{n+j+1}(k\mu)^{k(j+1)}}{(k(j+1)-1)!
	\Gamma(n+j+2)}\times\int_0^{t}p_0(z)(t-z)^{n+k+j(k+1)}e^{-(\lambda+k\mu)(t-z)}dz.
\end{split}
\end{align}
We can also give an equivalent representation of the queue by using the {\em queue length process}, that is to say the process $\cL(t)$ given by:
\begin{equation*}
\cL(t)=\begin{cases} k(N(t)-1)+S(t) & N(t)>0 \\
0 & N(t)=0.
\end{cases}
\end{equation*}
The queue length process gives us the length of the queue by means of phases. In particular, since the map $m_k:\cS \to \N_0$ given by:
\begin{equation}\label{eq:mk}
m_k(n,s)=\begin{cases} k(n-1)+s & (n,s)\in \cS^* \\
0 & (n,s)=(0,0)
\end{cases}
\end{equation}
is bijective, each state of $\cL(t)$ represents only one state of $Q(t)$ and vice-versa, so that we can equivalently describe the queue with $Q(t)$ or $\cL(t)$: indeed we have $m_k(Q(t))=\cL(t)$. Let us also recall that the inverse map of $m_k(n,s)$ is given by $(n_k(m),s_k(m))$ where:
\begin{equation}\label{eq:sk}
s_k(m)=\begin{cases}
\min\{s>0: \ s \in [m]_k\} & m>0 \\
0 & m=0
\end{cases}
\end{equation}
where $[m]_k$ is the residual class of $m$ modulo $k$ and:
\begin{equation}\label{eq:nk}
n_k(m)=\begin{cases}
\frac{m-s_k(m)}{k}+1 & m>0 \\
0 & m=0.
\end{cases}
\end{equation}
Let us define the state probabilities of $\cL(t)$ as 
\begin{equation*}
\cP_n(t)=\bP(\cL(t)=n|\cL(0)=0).
\end{equation*}
In \cite{Luchak1956} and \cite{Luchak1958} the forward equations for $\cP_n$ are given as:
\begin{equation}\label{eq:sysql}
\begin{cases}
\der{\cP_0}{t}(t)=-\lambda \cP_0(t)+k\mu \cP_1(t) \\
\der{\cP_n}{t}(t)=-(\lambda+k\mu) \cP_{n}(t)+k\mu \cP_{n+1}(t)+\lambda \sum_{m=1}^{n}c_m\cP_{n-m}(t) & n \ge 1\\
\cP_n(0)=\delta_{n,0} & n \ge 0
\end{cases}
\end{equation}
where $c_m=\delta_{m,k}$ for $M/E_k/1$ queues. However, the bijectivity of $m_k(n,s)$ gives us the equality
\begin{equation*}
\cP_{m}(t)=p_{n_k(m),s_k(m)}(t)
\end{equation*}
so that the difference-differential system \eqref{eq:sysql} is just a compact rewriting of \eqref{eq:syssp} where the states of $Q(t)$ are disposed in lexicographic order.\\
In \cite{Luchak1956} the mean queue length, defined as $M(t)=\E[\cL(t)|\cL(0)=0]$, is show to solve the following Cauchy problem:
\begin{equation*}
\begin{cases}
\der{M}{t}(t)=k(\lambda-\mu)+k\mu\cP_0(t)\\
M(0)=0
\end{cases}
\end{equation*}
which leads, in \cite{Luchak1958}, to the formula:
\begin{equation}\label{eq:meanql1}
M(t)=k(\lambda-\mu)t+k\mu \int_0^t\cP_0(y)dy.
\end{equation}
Moreover, in \cite{Luchak1958}, the distribution function of the busy period has been obtained as
\begin{equation}\label{eq:B1}
B(t)=\sum_{r=0}^{+\infty}\frac{k\lambda^r(k\mu)^{k(r+1)}}{r!\Gamma(rk+k+1)}\int_0^t z^{k+r(k+1)-1}e^{-(\lambda+k\mu)z}dz.
\end{equation}
Finally, let us also recall that the probability density function of the waiting time of a customer that enters the system at time $t>0$ has been given in \cite{Gaver1954} as:
\begin{equation}\label{eq:w1unco}
w(\xi;t)=\sum_{n=0}^{+\infty}\cP_n(t)k\mu\frac{(k\mu t)^{n-1}}{(n-1)!}e^{-k\mu t}.
\end{equation}
The author obtained such formula by using explicitly the Markov property of the queue.
\section{Fractional Erlang Queues}\label{sec3}
In \cite{CahoyPolitoPhoha2015} the authors introduce a fractional version of the $M/M/1$ queue. In particular, such queue has been shown useful to model some financial data for which classical birth-death processes were not enough. Following this idea, now we give the definition of a fractional $M/E_k/1$ queue. Let us consider a classical state-phase process $Q(t)=(N(t),S(t))$ for the $M/E_k/1$ queue and define $\{\sigma_\nu(t)\}_{t \ge 0}$ a $\nu$-stable subordinator independent from $Q(t)$, such that (see, for instace, \cite{MeerschaertSikorskii2011}) 
\begin{equation*}
\E[e^{-v\sigma_\nu(t)}]=e^{-tv^\nu}, \ v>0, \ \nu \in (0,1).
\end{equation*}
Define
\begin{equation*}
L_\nu(t)=\inf\{s \ge 0: \ \sigma_\nu(s)> t\} \quad t\ge 0
\end{equation*}
the inverse of the $\nu$-stable subordinator and pose $N^\nu(t):=N(L_\nu(t))$, $S^\nu(t):=S(L_\nu(t))$ and \begin{equation*}
Q^\nu(t):=(N^\nu(t),S^\nu(t))=Q(L_\nu(t)).
\end{equation*}
We will say that the process $Q^\nu(t)$ is the state-phase process of the fractional Erlang queue $M/E_k/1$. By definition, let us also pose $Q^1(t)=Q(t)$.\\
Before giving an interpretation in the sense of the queueing theory, let us show what are the forward equations for the state probabilities of $Q^\nu(t)$. Let us define:
\begin{align*}
p^\nu_{n,s}(t)&=\bP(Q^\nu(t)=(n,s)|Q^\nu(0)=(0,0)) & (n,s)\in \cS^* \\
p^\nu_0(t)&=\bP(Q^\nu(t)=(0,0)|Q^\nu(0)=(0,0)).
\end{align*}
To obtain such forward equations, we will need the fractional integral (see \cite{LiQianChen2011}) of order $\nu \in (0,1)$ defined, for a function $x:[0,t_1]\subseteq \R \to \R$, as
\begin{equation}
\cI_t^\nu x=\frac{1}{\Gamma(\nu)}\int_{0}^t(t-\tau)^{\nu-1}x(\tau)d\tau
\end{equation}
for any $t \in [0,t_1]$ and the Caputo fractional derivative of order $\nu \in (0,1)$
\begin{equation}
D^\nu_tx=\frac{1}{\Gamma(1-\nu)}\int_0^t\der{x}{t}(\tau)\frac{d\tau}{(t-\tau)^\nu}.
\end{equation}
The classes of functions for which such operators are well-defined are discussed in \cite{MeerschaertSikorskii2011}.\\
Moreover, we will use a Laplace transform approach, as described in \cite{KexuePeng2011,Podlubny1998}. Let us also remark that if $\overline{x}$ is the Laplace transform of the function $x$ and $\cL$ is the Laplace transform operator, then:
\begin{equation}\label{eq:LapDC}
\cL[D_t^\nu x](z)=z^\nu \overline{x}(z)-z^{\nu-1}x(0).
\end{equation}
For the state probabilities we can show the following theorem.
\begin{thm}\label{thm:fracsys}
The state probabilities $(p^\nu_{n,s})_{(n,s)\in \cS}$ are solution of the following fractional difference-differential Cauchy problem:
{\footnotesize \begin{equation}\label{eq:fracsys}
\begin{cases}
D_t^\nu p_0^\nu=-\lambda p_0^\nu(t)+k\mu p_{1,1}^\nu(t) \\
D_t^\nu p_{1,s}^\nu=-(\lambda+k\mu)p_{1,s}^\nu(t)+k\mu p_{1,s+1}^\nu(t) & 1 \le s \le k-1 \\
D_t^\nu p_{1,k}^\nu=-(\lambda+k\mu)p_{1,k}^\nu(t)+k\mu p_{2,1}^\nu(t)+\lambda p_0^\nu(t) \\
D_t^\nu p_{n,s}^\nu=-(\lambda+k\mu)p_{n,s}^\nu(t)+k\mu p_{n,s+1}^\nu(t)+\lambda p_{n-1,s}^\nu(t) & n \ge 2, \ 1 \le s \le k-1 \\
D_t^\nu p_{n,k}^\nu=-(\lambda+k\mu)p_{n,k}^\nu(t)+k\mu p_{n+1,1}^\nu(t)+\lambda p_{n-1,k}^\nu(t) & n \ge 2 \\
p_0^\nu(0)=1 \\
p_{n,s}^\nu(0)=0 & n \ge 1, \ 1 \le s \le k.
\end{cases}
\end{equation}}
where $D_t^\nu$ is the Caputo fractional derivative.
\end{thm}
\begin{proof}
Let us define the probability generating function
\begin{equation*}
G^\nu(z,t)=\sum_{n=1}^{+\infty}\sum_{s=1}^{k}z^{k(n-1)+s}p_{n,s}^\nu(t), \ |z|\le 1.
\end{equation*}
By multiplying the second equation of \eqref{eq:fracsys} by $z^{s+1}$, the third equation of \eqref{eq:fracsys} by $z^{k+1}$, the fourth equation of \eqref{eq:fracsys} by $z^{k(n-1)+s+1}$, the fifth equation of \eqref{eq:fracsys} by $z^{kn+1}$ and the summing all these equations with respect to $n$ and $s$ we have that the state probabilities $p^\nu_{n,s}$ solve this Cauchy problem if and only if $G^\nu(z,t)$ solves the following Cauchy problem:
\begin{equation}\label{eq:fracgenprob}
\begin{cases}
z{}D_t^\nu G^\nu(z,\cdot)=(1-z)[G^\nu(z,t)(k\mu-\lambda(z+\dots+z^k))-k\mu p_0^\nu(t)]\\
G^\nu(z,0)=0.
\end{cases}
\end{equation}
Denoting with $\widetilde{G}^\nu(z,v)$ and $\pi_0^\nu(v)$ the Laplace transform of $G^\nu(z,t)$ and $p_0^\nu(t)$, we know that $G^\nu(z,t)$ solves the Cauchy problem \eqref{eq:fracgenprob} if and only if its Laplace transform solves:
\begin{equation}\label{eq:tildeG}
z v^\nu \widetilde{G}^\nu(z,v)=(1-z)[\widetilde{G}^\nu(z,v)(k\mu-\lambda(z+\dots+z^k))-k\mu \pi_0^\nu(v)]
\end{equation}
where we used the formula for the Laplace transform of the Caputo derivative given in formula \eqref{eq:LapDC}.
Now let us observe that:
\begin{align}\label{eq:pnsnuint}
\begin{split}
&p_{n,s}^\nu(t)\!=\!\bP(Q^\nu(t)\!=\!(n,s)|Q^\nu(0)\!=\!(0,0))\!=\!\bP(Q(L_\nu(t))\!=\!(n,s)|Q(0)\!=\!(0,0))\\
&=\hspace*{-0.2cm}\int_{0}^{+\infty}\hspace*{-0.5cm}\bP(Q(y)\!=\!(n,s)|Q(0)=\!=\!(0,0))\bP(L_\nu(t)\in dy)\!=\!\int_0^{+\infty}p_{n,s}^1(t)\bP(L_\nu(t)\in dy)
\end{split}
\end{align}
and, analogously:
\begin{equation}\label{eq:p0}
p_0^\nu(t)=\int_0^{+\infty}p_0^1(y)\bP(L_\nu(t)\in dy).
\end{equation}
Thus for the probability generating function we know that
\begin{equation}\label{eq:Gnu}
G^\nu(z,t)=\int_0^{+\infty}G^1(z,y)\bP(L_\nu(t)\in dy).
\end{equation}
Let us recall from \cite{MeerschaertStraka2013} the following Laplace transform
\begin{equation}\label{eq:Lapinvstab}
\cL[\cP(L_\nu(t)\in dy)](v)=v^{\nu-1}e^{-yv^\nu}dy, \ v>0.
\end{equation}
Thus by using Eq. \eqref{eq:Lapinvstab} in Eqs. \eqref{eq:p0} and \eqref{eq:Gnu} we obtain:
\begin{align}
\pi_0^\nu(v)&=\int_0^{+\infty}p_0^1(y)v^{\nu-1}e^{-yv^\nu}dy\label{eq:pi0},
\widetilde{G}^\nu(z,v)&=\int_0^{+\infty}G^1(z,y)v^{\nu-1}e^{-yv^\nu}dy.
\end{align}
Using these two formulas we have that $\widetilde{G}^\nu(z,v)$ solves Eq. \eqref{eq:tildeG} if and only if
\begin{align}\label{eq:intlapeq}
\begin{split}
z v^\nu &\int_0^{+\infty}G^1(z,y)e^{-yv^\nu}dy=\\&=\int_0^{+\infty}(1-z)[G^1(z,y)(k\mu-\lambda(z+\dots+z^k))-k\mu p_0^1(y)]e^{-yv^\nu}dy.
\end{split}
\end{align}
From \cite{GrifLeoWill2005} we know that $G^1(z,t)$ is solution of the following Cauchy problem:
\begin{equation}\label{eq:genprob}
\begin{cases}
z\pd{G^1}{t}(z,t)=(1-z)[G^1(z,t)(k\mu-\lambda(z+\dots+z^k))-k\mu p_0^1(t)]\\
G^1(z,0)=0.
\end{cases}
\end{equation}
Using such property in \eqref{eq:intlapeq} we obtain
\begin{equation*}
zv^\nu \int_0^{+\infty}G^1(z,y)e^{-yv^\nu}dy=z\int_0^{+\infty}\pd{G^1(z,y)}{t}e^{-yv^\nu}dy
\end{equation*}
which, by using the integration by parts formula, gives us an identity.
\end{proof}
Let us introduce, as we did for the classical Erlang queue, the queue length process:
\begin{equation*}
\cL^\nu(t)=\begin{cases} 0 & N^\nu(t)=0\\
k(N^\nu(t)-1)+S^\nu(t) & N^\nu(t)\not = 0
\end{cases}
\end{equation*}
and define $\cP_m^\nu(t):=\bP(\cL^\nu(t)=m|\cL^\nu(0)=0)$. By the equality $\cP_m^\nu(t)=p^\nu_{n_k(m),s_k(m)}(t)$ and Theorem \ref{thm:fracsys}, we know that $(\cP_m^\nu(t))_{m \in \N}$ is solution of the fractional Cauchy problem
{\footnotesize \begin{equation}\label{eq:fracsysql}
\begin{cases}
{}D_t^\nu\cP_0^\nu=-\lambda \cP_0^\nu(t)+k\mu \cP_1^\nu(t) \\
{}D_t^\nu\cP_n^\nu=-(\lambda+k\mu) \cP_{n}^\nu(t)+k\mu \cP_{n+1}^\nu(t)+\lambda \sum_{m=1}^{n}c_m\cP_{n-m}^\nu(t) & n \ge 1\\
\cP_n^\nu(0)=\delta_{n,0} & n \ge 0
\end{cases}
\end{equation}}
where $c_j=\delta_{j,k}$, which is equivalent to the fractional Cauchy problem \eqref{eq:fracsys}. Since the solutions $\cP^\nu_m(t)$ are defined as probabilities, by using Corollary $2$ of \cite{AscioneLeonenkoPirozzi2018} and Schur's test (see \cite{HalmosSunder2012}) one can prove the following:
\begin{cor}\label{cor:unique}
The fractional Cauchy problems \eqref{eq:fracsys} and \eqref{eq:fracsysql} admit unique global solutions $\mathbf{p}^\nu(t)=(p_{n,s}^\nu(t))_{(n,s) \in \cS}$ and $\mathbf{P}^\nu(t)=(\cP_m^\nu(t))_{m \in \N}$, where $\mathbf{p}^\nu,\mathbf{P}^\nu:[0,+\infty) \to l^2$.
\end{cor}
%%%%%%%%%%%%%%%%%%%%%%%%%%%%%%%%%%%%%%%%%%%%%%%%%%%%%%%%%%%%%%%%%%%%
\section{Interpretation of the fractional $M/E_k/1$ queue}\label{sec4}
%%%%%%%%%%%%%%%%%%%%%%%%%%%%%%%%%%%%%%%%%%%%%%%%%%%%%%%%%%%%%%%%%%%%%%%
By using the result given in Theorem \ref{thm:fracsys} we can show what are the main features of a fractional $M/E_k/1$ queue in terms of interarrival and service times. Let us define:
\begin{align*}
h^\nu_n(t)&:=\bP(N^\nu(t)=n|Q^\nu(0)=(0,0)) \ n \ge 1\\
q^\nu_s(t)&:=\bP(S^\nu(t)=s|Q^\nu(0)=(0,0)) \ 1 \le s \le k.
\end{align*}
Note that
\begin{align*}
h^\nu_n(t)=\sum_{s=1}^{k}p^\nu_{n,s}(t) && q^\nu_s(t)=\sum_{n=1}^{+\infty}p^\nu_{n,s}(t).
\end{align*}
Let us work with the functions $h_n(t)$. By summing the equations of \eqref{eq:fracsys} with respect to $s$ we obtain the following fractional Cauchy problem
\begin{equation*}
\begin{cases}
D_t^\nu p_0^\nu=-\lambda p_0^\nu(t)+k\mu p_{1,1}^\nu(t) \\
D_t^\nu h_1^\nu=-\lambda h_1^\nu(t)+k\mu(p_{2,1}^\nu(t)-p_{1,1}^\nu(t))+\lambda p_0^\nu(t)\\
D_t^\nu h_n^\nu=-\lambda h_n^\nu(t)+k\mu(p_{n+1,1}^\nu(t)-p_{n,1}^\nu(t))+\lambda h_{n-1}^\nu(t) & n \ge 2\\
p_0^\nu(t)=1 \\
h_n^\nu(t)=0 & n \ge 1
\end{cases}
\end{equation*}
that, for $\nu=1$, gives us
\begin{equation*}
\begin{cases}
\der{p_0^1}{t}(t)=-\lambda p_0^1(t)+k\mu p_{1,1}^1(t) \\
\der{h_1^1}{t}(t)=-\lambda h_1^1(t)+k\mu(p_{2,1}^1(t)-p_{1,1}^1(t))+\lambda p_0^1(t)\\
\der{h_n^1}{t}(t)=-\lambda h_n^1(t)+k\mu(p_{n+1,1}^1(t)-p_{n,1}^1(t))+\lambda h_{n-1}^1(t) & n \ge 2\\
p_0^1(t)=1 \\
h_n^1(t)=0 & n \ge 1.
\end{cases}
\end{equation*}
Before proceeding, let us give some notation. Let $T$ be a random variable with distribution function given by
\begin{equation*}
F_T(t)=1-E_\nu(-\lambda t^\nu), \ t \ge 0, \ \lambda >0,
\end{equation*}
where $E_\nu$ is the Mittag-Leffler function (see \cite{KilbasSrivastavaTrujillo2006}) defined as
\begin{equation}\label{eq:ML1}
E_\nu(z)=\sum_{k=0}^{+\infty}\frac{z^k}{\Gamma(\nu k+1)}, \nu>0, z \in \C;
\end{equation}
we call $T$ a {\em Mittag-Leffler random variable} with fractional index $\nu$ and parameter $\lambda$ and we denote it by $T \sim ML_\nu(\lambda)$. In particular let us denote $\Psi(t)=E_\nu(-\lambda t^\nu)$, $f_T$ the probability density function of $T$ and 
%Then
%\begin{equation*}
%f_T(t)=-\der{\Psi(t)}{t}.
%\end{equation*}
%Thus, 
denoting with $\overline{f}_T$ its Laplace transform and recalling the Laplace transform of the Mittag-Leffler function that can be obtained from \eqref{eq:ML3LT} by posing $\delta=\gamma=1$, we have
\begin{equation}\label{eq:MLdensLT}
\overline{f}_T(v)=\frac{\lambda}{\lambda+v^\nu},\  v >0.
\end{equation}
Now we can show the following Theorem.
\begin{thm}\label{thm:interarr}
The interarrival times of a $\nu$-fractional $M/E_k/1$ queue are i.i.d. Mittag-Leffler random variables with fractional index $\nu$ and parameter $\lambda$.
\end{thm}
\begin{proof}
The independence of the interarrival times of a fractional $M/E_k/1$ queue follows easily from the definition. Let us show that these interarrival times are Mittag-Leffler distributed. To do this, let us introduce the process $\widetilde{N}^1(t)$ as a modification of $N^1(t)$ such that the state probabilities 
\begin{equation*}
\widetilde{h}^1_{n}(t):=\bP(\widetilde{N}^1(t)=n) \ n \ge 0
\end{equation*}
are solution of the following Cauchy problem
\begin{equation*}
\begin{cases}
\der{\widetilde{h}_{m}^1}{t}(t)=-\lambda \widetilde{h}_{m}^1(t) \\
\der{\widetilde{h}_{m+1}^1}{t}(t)=\lambda \widetilde{h}_{m}^1\\
\der{\widetilde{h}_n^1}{t}(t)=0 \ n\ge 0, n \not = m,m+1 \\
\widetilde{h}_n^1(0)=\delta_{n,m}
\end{cases}
\end{equation*}
that is to say that $\widetilde{N}^1(t)$ is a birth process ($\mu=0$) that has the same birth rate $\lambda$ as $N^1(t)$, starts at $m$ ad admits $m+1$ as absorbent state. In particular the Cauchy problem can be reduced to a couple of linear ODEs:
\begin{equation*}
\begin{cases}
\der{\widetilde{h}_m^1}{t}(t)=-\lambda \widetilde{h}_m^1(t) \\
\der{\widetilde{h}_{m+1}^1}{t}(t)=\lambda \widetilde{h}_m^1(t)\\
\widetilde{h}_m^1(0)=1 \\
\widetilde{h}_{m+1}^1(0)=0.
\end{cases}
\end{equation*}
Moreover, since they are probabilities, we have also that $\widetilde{h}_m^1(t)+\widetilde{h}_{m+1}^1(t)=1$ for any $t \ge 0$. Finally, denoting with $T^1$ the arrival time of the first new customer of the queue $Q^1$ starting from $N^1(0)=m$, we have that $\widetilde{h}_{m+1}^1(t)=F_{T^1}(t)$. \\
Now let us define $\widetilde{N}^\nu(t)=\widetilde{N}^1(L_\nu(t))$ and
\begin{align*}
\widetilde{h}^\nu_{m}(t):=\bP(\widetilde{N}^\nu(t)=m) &&
\widetilde{h}^\nu_{m+1}(t):=\bP(\widetilde{N}^\nu(t)=m+1).
\end{align*}
By definition, $\widetilde{N}^\nu(t)$ is a process that starts from $m$ and admits $m+1$ as absorbent state. Moreover, denoting with $T^\nu$ the arrival time of the first new customer in the queue $Q^\nu$ starting from $N^\nu(0)=m$, we have that $\widetilde{h}_{m+1}^\nu(t)=F_{T^\nu}(t)$. Indeed, since $N^1(t)$ is a Markov process, $N^\nu(t)$ is a semi-Markov process and, denoting with $T_n$ the $n$-th jump time, the discrete time process $N^\nu(T_n)$ is a time-homogeneous Markov chain (see, for instance, \cite{GihmanSkorohod1975}). For this reason, $T^\nu$ can be seen as the interarrival time of the queue $Q^\nu$.\\
Now we have to show that
\begin{equation}\label{eq:intarrpass1}
\begin{cases}
D_t^\nu\widetilde{h}_m^\nu=-\lambda \widetilde{h}_m^\nu(t) \\
D_t^\nu\widetilde{h}_{m+1}^\nu=\lambda \widetilde{h}_m^\nu(t)\\
\widetilde{h}_m^\nu(0)=1 \\
\widetilde{h}_{m+1}^\nu(0)=0.
\end{cases}
\end{equation}
The initial conditions are obviously verified. Now let us work with the first equation. We know that this equation is verified if and only if, denoting with $\overline{h}_m^\nu$ the Laplace transform of $\widetilde{h}_m^\nu$:
\begin{equation}\label{eq:intarrpass2}
v^\nu \overline{h}_m^\nu(v)-v^{\nu-1}=-\lambda \overline{h}_m^\nu(v), \ v>0.
\end{equation}
By using the definition of $\widetilde{h}_m^\nu(t)$ we have that:
\begin{equation*}
\widetilde{h}_m^\nu(t)=\int_0^{+\infty}\widetilde{h}_m^1(y)\bP(L_\nu(t)\in dy)
\end{equation*}
and thus, taking the Laplace transform
\begin{equation*}
\overline{h}_m^\nu(v)=v^{\nu-1}\int_0^{+\infty}\widetilde{h}_m^1(y)e^{-yv^\nu}dy.
\end{equation*}
Using this formula, we know that Eq. \eqref{eq:intarrpass2} is verified if and only if
\begin{equation}\label{eq:intarrpass3}
v^\nu \int_0^{+\infty} \widetilde{h}^1_m(y)e^{-yv^\nu}dy-1=\int_{0}^{+\infty}-\lambda \widetilde{h}^1_m(y)e^{-yv^\nu}dy.
\end{equation}
Recalling that $\widetilde{h}^1_m(t)$ is solution of the following Cauchy problem
\begin{equation*}
\begin{cases}
\der{\widetilde{h}^1_m}{t}(t)=-\lambda \widetilde{h}^1_m(t)\\
\widetilde{h}^1_m(0)=1
\end{cases}
\end{equation*}
we have that Eq. \eqref{eq:intarrpass3} is verified if and only if
\begin{equation*}
v^\nu \int_0^{+\infty}\widetilde{h}^1_m(y)e^{-yv^\nu}dy-1=\int_0^{+\infty}\der{\widetilde{h}^1_m(y)}{t}e^{-yv^\nu}dy
\end{equation*}
that, by using the integration by parts formula, gives us an identity. In an analogous way, one can show that the second equation of \eqref{eq:intarrpass1} is verified.\\
Finally, we only have to solve the fractional Cauchy problem \eqref{eq:intarrpass1}. To do this, let us observe that the first equation is independent from the second one, so we can solve first the Cauchy problem:
\begin{equation*}
\begin{cases}
D_t^\nu\widetilde{h}_m^\nu=-\lambda \widetilde{h}_m^\nu(t) \\
\widetilde{h}_m^\nu(0)=1 \\
\end{cases}
\end{equation*}
whose solution is $\widetilde{h}_m^\nu(t)=E_\nu(-\lambda t^\nu)$ (see \cite{KilbasSrivastavaTrujillo2006}). Finally:
\begin{equation*}
F_{T^\nu}(t)=\widetilde{h}_{m+1}^\nu(t)=1-\widetilde{h}_m^\nu(t)=1-E_\nu(-\lambda t^\nu), \ t \ge 0.
\end{equation*}
\end{proof}
Now let us work with the functions $q_s^\nu(t)$, in order to obtain some information on the service times. Summing the equations in \eqref{eq:fracsys} with respect to $n$ we have:
\begin{equation*}
\begin{cases}
D_t^\nu p_0^\nu=-\lambda p_0^\nu(t) + k \mu p_{1,1}^\nu(t)\\
D_t^\nu q_s^\nu=-k\mu q_s^\nu(t)+k\mu q_{s+1}^\nu(t) & 1 \le s \le k-1 \\
D_t^\nu q_k^\nu=-k\mu q_k^\nu(t)+k\mu q_1^\nu(t)-k\mu p_{1,1}^\nu(t)+\lambda p_0^\nu(t)\\
p_0^\nu(0)=1\\
q_s^\nu(0)=0 & 1 \le s \le k
\end{cases}
\end{equation*}
that, for $\nu=1$, gives us
\begin{equation*}
\begin{cases}
\der{p_0^1}{t}(t)=-\lambda p_0^1(t) + k \mu p_{1,1}^1(t)\\
\der{q_s^1}{t}(t)=-k\mu q_s^1(t)+k\mu q_{s+1}^1(t) & 1 \le s \le k-1 \\
\der{q_k^1}{t}(t)=-k\mu q_k^1(t)+k\mu q_1^1(t)-k\mu p_{1,1}^1(t)+\lambda p_0^1(t)\\
p_0^1(0)=1\\
q_s^1(0)=0 & 1 \le s \le k.
\end{cases}
\end{equation*}
We can use the functions $q_s^\nu$ exactly in the same way we did with $h_n^\nu$ to show the following Theorem.
\begin{thm}\label{thm:interph}
The inter-phase times of a $\nu$-fractional $M/E_k/1$ queue are i.i.d. Mittag-Leffler random variables with fractional index $\nu$ and parameter $k\mu$.
\end{thm}
Let us give some other notation. Let $T$ be a random variable with distribution function given by
{\small \begin{equation*}
F_T(t)=1-\sum_{n=0}^{k-1}\frac{(\lambda t^\nu)^n}{n!}E_\nu^{(n)}(-\lambda t^\nu)=1-\sum_{n=0}^{k-1}\sum_{j=1}^{+\infty}(-1)^j\binom{n+j}{n}\frac{(\lambda t^\nu)^{j+n}}{\Gamma(\nu(j+n)+1)}, \ t \ge 0;
\end{equation*}}
we call $T$ a {\em generalized Erlang random variable} (as introduced in \cite{MainardiGorenfloScalas2007}) with fractional index $\nu$, shape parameter $k$ and rate $\lambda$ and we denote it by $T \sim GE_\nu(k,\lambda)$. In particular $T\sim GE_\nu(k,\lambda)$ is sum of $k$ i.i.d. Mittag-Leffler random variable of fractional index $\nu$ and parameter $\lambda$. For this reason, denoting with $f_T$ the probability density function of $T$ and $\overline{f}_T$ its Laplace transform, we have, from \eqref{eq:MLdensLT},
\begin{equation*}
\overline{f}_T(v)=\frac{\lambda^k}{(\lambda+v^\nu)^k}, \ v>0.
\end{equation*}
From Theorem \ref{thm:interph} we easily obtain the following Corollary
\begin{cor}\label{cor:servtime}
The service times of a $\nu$-fractional $M/E_k/1$ queue are i.i.d. generalized Erlang random variables with fractional index $\nu$, shape parameter $k$ and rate $k\mu$.
\end{cor}
\begin{proof}
Let us simply observe that a service time is sum of $k$ inter-phase times, which are i.i.d. Mittag-Leffler random variables of fractional index $\nu$ and parameter $k\mu$.
\end{proof}
From Theorems \ref{thm:interarr} and \ref{thm:interph} and Corollary \ref{cor:servtime} we can conclude that a fractional $M/E_k/1$ queue is a single-channel queue with Mittag-Leffler interarrival times and generalized Erlang service times in which the service is given through $k$ phases, each one with Mittag-Leffler distributed service time. In particular for $k=1$ we obtain the fractional $M/M/1$ queue introduced in \cite{CahoyPolitoPhoha2015}.\\
We can also obtain the distribution of a general sojourn time of $\cL^\nu(t)$ working with $\cP_n^\nu(t)$ as we already did with $q_s^\nu(t)$ and $h_n^\nu(t)$.
\begin{thm}
The sojourn time of the queue length process of a $\nu$-fractional $M/E_k/1$ queue in a state $\overline{n}>0$ is a Mittag-Leffler random variable with fractional index $\nu$ and parameter $\lambda+k\mu$.
\end{thm}
\begin{rmk}
The sojourn time of $\cL^\nu(t)$ in the state $0$ coincides with an interarrival time, so it is a Mittag-Leffler random variable of fractional index $\nu$ and parameter $\lambda$.
\end{rmk}
From this information on sojourn time, we can also obtain a new information on the mutual dependence of interarrival times and inter-phase times.
\begin{cor}
Interarrival times and inter-phase times are not independent for $\nu \in (0,1)$.
\end{cor}
\begin{proof}
Let us first remark that if two random variables $T$ and $S$ are independent, then
\begin{equation*}
\bP(\min\{T,S\}\ge t)=\bP(T \ge t, S \ge t)=\bP(T \ge t)\bP(S \ge t).
\end{equation*}
Let $T_S$ be a sojourn time, $T_A$ an interarrival time and $T_P$ an inter-phase time. Thus we have $T_S=\min\{T_A,T_P\}$. However
{\small \begin{align*}
\bP(T_S\ge t)=E_\nu(-(\lambda+k\mu)t^\nu), &&
\bP(T_P\ge t)=E_\nu(-k\mu t^\nu), &&
\bP(T_A\ge t)=E_\nu(-\lambda t^\nu)
\end{align*}}
thus we have that $\bP(T_S \ge t)=\bP(T_P \ge t)\bP(T_A \ge t)$ if and only if
\begin{equation*}
E_\nu(-(\lambda+k\mu)t^\nu)=E_\nu(-k\mu t^\nu)E_\nu(-\lambda t^\nu)
\end{equation*}
that is true if and only if $\lambda=k\mu=0$ or $\nu=1$ (see \cite{KexuePeng2010}).
\end{proof}
\section{Transient state probabilities}\label{sec5}
Our aim is now to obtain the transient state probability functions $p_0^\nu(t)$ and $p_{n,s}^\nu(t)$ in a closed form, with the aid of the three-parameter Mittag-Leffler function $E_{\nu,\mu}^\delta$ (see \cite{KilbasSrivastavaTrujillo2006}) defined as
\begin{equation}\label{eq:ML3}
E_{\nu,\mu}^{\delta}(z)=\sum_{k=0}^{+\infty}\frac{\Gamma(\delta+k)z^k}{k!\Gamma(\delta)\Gamma(\nu k+\mu)}, \ \Re(\gamma),\Re(\nu),\Re(\delta)>0, \ z \in \C
\end{equation}
and the Laplace transform formula (see \cite{HauboldMathaiSaxena2011})
\begin{multline}\label{eq:ML3LT}
\cL[z^{\gamma-1}E_{\nu,\gamma}^\delta(wz^\nu)](v)=\frac{v^{\nu,\delta-\gamma}}{(v^\nu-w)^{\delta}}, \\ \gamma,\nu,\delta,w \in \C, \ \Re(\gamma),\Re(\nu),\Re(\delta)>0, \ v \in \C, |wv^\nu|<1.
\end{multline}
To do this, let us first work with $p_0^\nu(t)$.
\begin{thm}
Let $p_0^\nu(t)$ be the $0$ state probability of a fractional $M/E_k/1$ queue $Q^\nu(t)$ and $\pi_0^\nu(v)$ its Laplace transform. Then we have
\begin{equation}\label{eq:p0nu}
p_0^\nu(t)=\sum_{m=1}^{+\infty}\sum_{r=0}^{+\infty}C^0_{m,r}t^{\gamma^0_{m,r}-1}E_{\nu,\gamma^0_{m,r}}^{\delta^0_{m,r}}(-(\lambda+k\mu)t^\nu)
\end{equation}
and
\begin{equation}\label{eq:pi0nu}
\pi_0^\nu(v)=\sum_{m=1}^{+\infty}\sum_{r=0}^{+\infty}C^0_{m,r}\frac{v^{\nu\delta^0_{m,r}-\gamma^0_{m,r}}}{(\lambda+k\mu+v^\nu)^{\delta^0_{m,r}}}
\end{equation}
where
\begin{equation}\label{eq:C0mr}
\begin{gathered}
C^0_{m,r}=\frac{m}{m+r(k+1)}\binom{m+r(k+1)}{m+rk}\lambda^r(k\mu)^{m+rk-1},\\
\delta^0_{m,r}=m+r(k+1),  \qquad \gamma^0_{m,r}=\nu(\delta^0_{m,r}-1)+1.
\end{gathered}
\end{equation}
\end{thm}
\begin{proof}
Let us recall that we already have a closed form for $p_0^1(t)$, given by Eq. \eqref{eq:p01}. Moreover, using this closed form in Eq. \eqref{eq:pi0} we have:
\begin{equation*}
\pi_0^\nu(v)=\sum_{m=1}^{+\infty}\sum_{r=0}^{+\infty}\frac{m\lambda^r(k\mu)^{m+rk-1}}{r!\Gamma(m+rk+1)}v^{\nu-1}\int_0^{+\infty}y^{m+r(k+1)-1}e^{-(\lambda+k\mu+v^\nu)y}dy
\end{equation*}
that, integrating, becomes
\begin{equation*}
\pi_0^\nu(v)=\sum_{m=1}^{+\infty}\sum_{r=0}^{+\infty}\frac{m\lambda^r(k\mu)^{m+rk-1}(m+r(k+1)-1)!}{r!\Gamma(m+rk+1)}\frac{v^{\nu-1}}{(\lambda+k\mu+v^\nu)^{m+r(k+1)}}.
\end{equation*}
Now let us pose $C^0_{m,r}$, $\delta^0_{m,r}$ and $\gamma^0_{m,r}$ as in \eqref{eq:C0mr} and observe that
\begin{equation*}
\pi_0^\nu(v)=\sum_{m=1}^{+\infty}\sum_{r=0}^{+\infty}C^0_{m,r}\frac{v^{\nu\delta^0_{m,r}-\gamma^0_{m,r}}}{(\lambda+k\mu+v^\nu)^{\delta^0_{m,r}}}.
\end{equation*}
Finally, by using Eq. \eqref{eq:ML3LT} we obtain:
\begin{equation*}
p_0^\nu(t)=\sum_{r=0}^{+\infty}\sum_{m=1}^{+\infty}C^0_{m,r}t^{\gamma^0_{m,r}-1}E_{\nu,\gamma^0_{m,r}}^{\delta^0_{m,r}}(-(\lambda+k\mu)t^\nu).
\end{equation*}
\end{proof}
\begin{rmk}
For $k=1$ we have that:
\begin{equation*}
\begin{gathered}
C^0_{m,r}=\frac{m}{m+2r}\binom{m+2r}{m+r}\lambda^r\mu^{m+r-1},\\
\delta^0_{m,r}=m+2r, \qquad
\gamma^0_{m,r}=\nu(m+2r-1)+1
\end{gathered}
\end{equation*}
and then we have
\begin{equation*}
p_0^\nu(t)=\sum_{r=0}^{+\infty}\frac{m}{m+2r}\binom{m+2r}{m+r}\lambda^r\mu^{m+r-1}t^{\nu(m+2r-1)}E^{m+2r}_{\nu,\nu(m+2r-1)+1}(-(\lambda+\mu)t^\nu).
\end{equation*}
Now, let us change variables, posing $\overline{m}=m+r$, so we have
\begin{equation*}
p_0^\nu(t)=\sum_{r=0}^{+\infty}\sum_{\overline{m}=r+1}^{+\infty}\frac{\overline{m}-r}{\overline{m}+r}\binom{\overline{m}+r}{\overline{m}}\lambda^r\mu^{\overline{m}-1}t^{\nu(\overline{m}+r-1)}E^{\overline{m}+r}_{\nu,\nu(\overline{m}+r-1)+1}(-(\lambda+\mu)t^\nu).
\end{equation*}
This is the second formula obtained in Remark $4$ in \cite{AscioneLeonenkoPirozzi2018} for the fractional $M/M/1$ queue. 
\end{rmk}
\begin{rmk}
For $\nu=1$ we obtain
\begin{equation*}
\gamma^0_{m,r}=\delta^0_{m,r}=m+r(k+1)
\end{equation*}
and then, with some calculations
\begin{align*}
\begin{split}
p_0^1(t)&=
%\sum_{m=1}^{+\infty}\sum_{r=0}^{+\infty}C^0_{m,r}t^{\delta^0_{m,r}-1}E_{\nu,\delta^0_{m,r}}^{\delta^0_{m,r}}(-(\lambda+k\mu)t^\nu)\\
\sum_{m=1}^{+\infty}\sum_{r=0}^{+\infty}\frac{C^0_{m,r}}{\Gamma(\delta^0_{m,r})}t^{\delta^0_{m,r}-1}e^{-(\lambda+k\mu)t}.
\end{split}
\end{align*}
Let us remark that
\begin{align*}
\begin{split}
\frac{C^0_{m,r}}{\Gamma(\delta^0_{m,r})}&=\frac{m\lambda^r(k\mu)^{m+rk-1}(m+r(k+1)-1)!}{r!\Gamma(m+rk+1)\Gamma(m+r(k+1))}
=\frac{m\lambda^r(k\mu)^{m+rk-1}}{r!\Gamma(m+rk+1)}
\end{split}
\end{align*}
so that
\begin{equation*}
p_0^1(t)=\sum_{r=0}^{+\infty}\frac{m\lambda^r(k\mu)^{m+rk-1}}{r!\Gamma(m+rk+1)}t^{m+r(k+1)-1}e^{-(\lambda+k\mu)t}
\end{equation*}
that is Eq. \eqref{eq:p01}.
\end{rmk}
To work with $p_{n,s}^\nu(t)$, it could be useful to show the following Lemma.
\begin{lem}
For any $n \in \N$ we have
\begin{align}\label{eq:teclemma}
\begin{split}
\int_0^{+\infty}&\int_0^{y}p_0^1(z)(y-z)^ne^{-(\lambda+k\mu)(y-z)}e^{-yv^\nu}dzdy\\&=\sum_{m=1}^{+\infty}\sum_{r=0}^{+\infty}C^0_{m,r}\frac{n!}{(\lambda+k\mu+v^\nu)^{m+r(k+1)+n+1}}
\end{split}
\end{align}
where $C^0_{m,r}$ is defined in \eqref{eq:C0mr}.
\end{lem}
\begin{proof}
Let us first notice that:
\begin{multline*}
\int_0^{+\infty}\int_0^{y}p_0^1(z)(y-z)^ne^{-(\lambda+k\mu)(y-z)}e^{-yv^\nu}dzdy=\\=\int_0^{+\infty}\int_0^{+\infty}p_0^1(z)(y-z)^ne^{-(\lambda+k\mu+v^\nu)(y-z)}e^{-zv^\nu}\chi_{[0,y]}(z)dzdy
\end{multline*}
and then we can use Fubini's theorem to obtain
\begin{multline*}
\int_0^{+\infty}\int_0^{+\infty}p_0^1(z)(y-z)^ne^{-(\lambda+k\mu+v^\nu)(y-z)}e^{-zv^\nu}\chi_{[0,y]}(z)dzdy=\\=\int_0^{+\infty}\int_0^{+\infty}p_0^1(z)(y-z)^ne^{-(\lambda+k\mu+v^\nu)(y-z)}e^{-zv^\nu}\chi_{[0,y]}(z)dydz.
\end{multline*}
Now let us pose $w=y-z$ to obtain
\begin{multline*}
\int_0^{+\infty}\int_0^{+\infty}p_0^1(z)(y-z)^ne^{-(\lambda+k\mu+v^\nu)(y-z)}e^{-zv^\nu}\chi_{[0,y]}(z)dydz\\=\int_0^{+\infty}\int_{-z}^{+\infty}p_0^1(z)w^ne^{-(\lambda+k\mu+v^\nu)w}e^{-zv^\nu}\chi_{[0,z+w]}(z)dwdz
\end{multline*}
that, since $\chi_{[0,z+w]}(z)=\chi_{[0,+\infty)}(w)$, gives us
\begin{multline*}
\int_0^{+\infty}\int_{-z}^{+\infty}p_0^1(z)w^ne^{-(\lambda+k\mu+v^\nu)w}e^{-zv^\nu}\chi_{[0,z+w]}(z)dwdz=\\=\left(\int_0^{+\infty}p_0^1(z)e^{-zv^\nu}dz\right)\left(\int_{0}^{+\infty}w^ne^{-(\lambda+k\mu+v^\nu)w}dw\right).
\end{multline*}
By integrating, we have
\begin{equation*}
\int_{0}^{+\infty}w^ne^{-(\lambda+k\mu+v^\nu)w}dw=\frac{n!}{(\lambda+k\mu+v^\nu)^{n+1}}.
\end{equation*}
Moreover, by \eqref{eq:pi0} and \eqref{eq:pi0nu} we have
\begin{align*}
\begin{split}
\int_0^{+\infty}&p_0^1(z)e^{-zv^\nu}dz=\frac{\pi_0^\nu(v)}{v^{\nu-1}}
=\sum_{m=1}^{+\infty}\sum_{r=0}^{+\infty}\frac{C^0_{m,r}}
{(\lambda+k\mu+v^\nu)^{m+r(k+1)}}
\end{split}
\end{align*}
so finally we obtain
\begin{align*}
\begin{split}
\int_0^{+\infty}&\int_0^{y}p_0^1(z)(y-z)^ne^{-(\lambda+k\mu)(y-z)}e^{-yv^\nu}dzdy\\&=\sum_{m=1}^{+\infty}\sum_{r=0}^{+\infty}C^0_{m,r}\frac{n!}{(\lambda+k\mu+v^\nu)^{m+r(k+1)+n+1}}.
\end{split}
\end{align*}
\end{proof}
Now we are ready to show the following Theorem.
\begin{thm}
For any $(n,s)\in \cS^*$ let $p_{n,s}^\nu(t)=\bP(Q^\nu(t)=(n,s)|Q^\nu(0)=(0,0))$ and $\pi_{n,s}^\nu(v)$ be its Laplace transform. Then
\begin{align*}
\begin{split}
p_{n,s}^\nu(t)&=\sum_{j=0}^{+\infty}A^{n,s}_jt^{\alpha_j^{n,s}-1}E^{a^{n,s}_j}_{\nu,\alpha_j^{n,s}}(-(\lambda+k\mu)t^\nu)\\ &+\sum_{j=0}^{+\infty}\sum_{m=1}^{+\infty}\sum_{r=0}^{+\infty}B^{n,s}_{j,m,r}t^{\beta_{j,m,r}^{n,s}-1}E^{b^{n,s}_{j,m,r}}_{\nu,\beta_{j,m,r}^{n,s}}(-(\lambda+k\mu)t^\nu)\\ &-\sum_{j=0}^{+\infty}\sum_{m=1}^{+\infty}\sum_{r=0}^{+\infty}C^{n,s}_{j,m,r}t^{\gamma_{j,m,r}^{n,s}-1}E^{c^{n,s}_{j,m,r}}_{\nu,\gamma_{j,m,r}^{n,s}}(-(\lambda+k\mu)t^\nu)
\end{split}
\end{align*}
and
\begin{align*}
\begin{split}
\pi_{n,s}^\nu(v)&=\sum_{j=0}^{+\infty}A^{n,s}_j\frac{v^{\nu a^{n,s}_j-\alpha^{n,s}_j}}{(\lambda+k\mu+v^\nu)^{a^{n,s}_j}} +\sum_{j=0}^{+\infty}\sum_{m=1}^{+\infty}\sum_{r=0}^{+\infty}B^{n,s}_{j,m,r}\frac{v^{\nu b^{n,s}_{j,m,r}-\beta^{n,s}_{j,m,r}}}{(\lambda+k\mu+v^\nu)^{b^{n,s}_{j,m,r}}}\\ &-\sum_{j=0}^{+\infty}\sum_{m=1}^{+\infty}\sum_{r=0}^{+\infty}C^{n,s}_{j,m,r}\frac{v^{\nu c^{n,s}_{j,m,r}-\gamma^{n,s}_{j,m,r}}}{(\lambda+k\mu+v^\nu)^{c^{n,s}_{j,m,r}}}
\end{split}
\end{align*}
where
{\small \begin{equation}\label{eq:const}
\begin{gathered}
A^{n,s}_{j}=\binom{n+k+kj+j-s}{n+j}\lambda^{n+j}(k\mu)^{k(j+1)-s}\\
a^{n,s}_j=n-s+(j+1)(k+1) \qquad \alpha^{n,s}_j=\nu(a^{n,s}_j-1)+1\\
B^{n,s}_{j,m,r}=k\mu C^0_{m,r}A^{n,s}_j \qquad b^{n,s}_{j,m,r}=\delta^0_{m,r}+a^{n,s}_j \qquad \beta^{n,s}_{j,m,r}=\nu(b^{n,s}_{j,m,r}-1)+1\\
C^{n,s}_{j,m,r}=\begin{cases}
k\mu C^0_{m,r}A^{n,s+1}_j & s\not = k\\
k\mu C^0_{m,r}A^{n+1,1}_j & s= k
\end{cases}
\qquad c^{n,s}_{j,m,r}=\begin{cases}\delta^0_{m,r}+a^{n,s+1}_j & s \not = k\\
\delta^0_{m,r}+a^{n+1,1}_j & s=k
\end{cases}\\
\gamma^{n,s}_{j,m,r}=\nu(c^{n,s}_{j,m,r}-1)+1.
\end{gathered}
\end{equation}}
\end{thm}
\begin{proof}
From Eq. \eqref{eq:pnsnuint} we easily obtain
\begin{equation}\label{eq:pins}
\pi_{n,s}^\nu(v)=v^{\nu-1}\int_0^{+\infty}p^1_{n,s}(y)e^{-yv^\nu}dy, \ v>0
\end{equation}
Let us first consider $1 \le s \le k-1$; by using Eq. \eqref{eq:pns1}, we have
\begin{align*}
\begin{split}
\pi_{n,s}^\nu(v)&\!=\!\sum_{j=0}^{+\infty}\frac{\lambda^{n+j}(k\mu)^{k(j+1)-s}}{(k(j+1)-s)!\Gamma(n+j+1)}v^{\nu-1}
\int_{0}^{+\infty}\hspace*{-0.5cm}y^{n+k-s+j(k+1)}e^{-(\lambda+k\mu+v^\nu)y}dy\\
&+\sum_{j=0}^{+\infty}\frac{\lambda^{n+j}(k\mu)^{k(j+1)-s+1}}{(k(j+1)-s)!\Gamma(n+j+1)}v^{\nu-1} \\ &\times \int_0^{+\infty}\int_0^{y}p_0^1(z)(y-z)^{n+k-s+j(k+1)}e^{-(\lambda+k\mu)(y-z)}e^{-yv^\nu}dzdy\\&-\sum_{j=0}^{+\infty}\frac{\lambda^{n+j}(k\mu)^{k(j+1)-s}}{(k(j+1)-s-1)!\Gamma(n+j+1)}v^{\nu-1} \\ & \times\int_0^{+\infty}\int_0^{y}p_0^1(z)(y-z)^{n+k-s-1+j(k+1)}e^{-(\lambda+k\mu)(y-z)}e^{-yv^\nu}dzdy.
\end{split}
\end{align*}
Integrating, using Eq. \eqref{eq:teclemma} and posing $A^{n,s}_j$, $a^{n,s}_j$, $\alpha^{n,s}_j$, $B^{n,s}_{j,m,r}$, $b^{n,s}_{j,m,r}$, $\beta^{n,s}_{j,m,r}$, $C^{n,s}_{j,m,r}$, $c^{n,s}_{j,m,r}$ and $\gamma^{n,s}_{j,m,r}$ as in equation \eqref{eq:const} for $s \not = k$ we obtain
\begin{align*}
\begin{split}
\pi_{n,s}^\nu(v)&=\sum_{j=0}^{+\infty}A^{n,s}_j\frac{v^{\nu a^{n,s}_j-\alpha^{n,s}_j}}{(\lambda+k\mu+v^\nu)^{a^{n,s}_j}} +\sum_{j=0}^{+\infty}\sum_{m=1}^{+\infty}\sum_{r=0}^{+\infty}B^{n,s}_{j,m,r}\frac{v^{\nu b^{n,s}_{j,m,r}-\beta^{n,s}_{j,m,r}}}{(\lambda+k\mu+v^\nu)^{b^{n,s}_{j,m,r}}}\\ &-\sum_{j=0}^{+\infty}\sum_{m=1}^{+\infty}\sum_{r=0}^{+\infty}C^{n,s}_{j,m,r}\frac{v^{\nu c^{n,s}_{j,m,r}-\gamma^{n,s}_{j,m,r}}}{(\lambda+k\mu+v^\nu)^{c^{n,s}_{j,m,r}}}.
\end{split}
\end{align*}
Finally, by using Eq. \eqref{eq:ML3LT} we obtain
\begin{align*}
\begin{split}
p_{n,s}^\nu(t)&=\sum_{j=0}^{+\infty}A^{n,s}_jt^{\alpha_j^{n,s}-1}E^{a^{n,s}_j}_{\nu,\alpha_j^{n,s}}(-(\lambda+k\mu)t^\nu)\\ &+\sum_{j=0}^{+\infty}\sum_{m=1}^{+\infty}\sum_{r=0}^{+\infty}B^{n,s}_{j,m,r}t^{\beta_{j,m,r}^{n,s}-1}E^{b^{n,s}_{j,m,r}}_{\nu,\beta_{j,m,r}^{n,s}}(-(\lambda+k\mu)t^\nu)\\ &-\sum_{j=0}^{+\infty}\sum_{m=1}^{+\infty}\sum_{r=0}^{+\infty}C^{n,s}_{j,m,r}t^{\gamma_{j,m,r}^{n,s}-1}E^{c^{n,s}_{j,m,r}}_{\nu,\gamma_{j,m,r}^{n,s}}(-(\lambda+k\mu)t^\nu).
\end{split}
\end{align*}
Now, for $s=k$, from Eq. \eqref{eq:pins} and \eqref{eq:pnk1} we obtain
\begin{align*}
\begin{split}
\pi_{n,k}^\nu(v)&=\sum_{j=0}^{+\infty}\frac{\lambda^{n+j}(k\mu)^{kj}}{(kj)!\Gamma(n+j+1)}v^{\nu-1}
\int_{0}^{+\infty}y^{n+j(k+1)}e^{-(\lambda+k\mu+v^\nu)y}dy\\
&+\sum_{j=0}^{+\infty}\frac{\lambda^{n+j}(k\mu)^{kj+1}}{(kj)!\Gamma(n+j+1)}v^{\nu-1} \\ &\times \int_0^{+\infty}\int_0^{y}p_0^1(z)(y-z)^{n+j(k+1)}e^{-(\lambda+k\mu)(y-z)}e^{-yv^\nu}dzdy\\&-\sum_{j=0}^{+\infty}\frac{\lambda^{n+j+1}(k\mu)^{k(j+1)}}{(k(j+1)-1)!\Gamma(n+j+2)}v^{\nu-1} \\ & \times\int_0^{+\infty}\int_0^{y}p_0^1(z)(y-z)^{n+k+j(k+1)}e^{-(\lambda+k\mu)(y-z)}e^{-yv^\nu}dzdy.
\end{split}
\end{align*}
Integrating, using Eq. \eqref{eq:teclemma} and posing $A^{n,s}_j$, $a^{n,s}_j$, $\alpha^{n,s}_j$, $B^{n,s}_{j,m,r}$, $b^{n,s}_{j,m,r}$, $\beta^{n,s}_{j,m,r}$, $C^{n,s}_{j,m,r}$, $c^{n,s}_{j,m,r}$ and $\gamma^{n,s}_{j,m,r}$ as in equation \eqref{eq:const} for $s = k$ we obtain
\begin{align*}
\begin{split}
\pi_{n,k}^\nu(v)&=\sum_{j=0}^{+\infty}A^{n,k}_j\frac{v^{\nu a^{n,k}_j-\alpha^{n,k}_j}}{(\lambda+k\mu+v^\nu)^{a^{n,k}_j}}\\ &\hspace*{-1cm}\!\!+\!\!\sum_{j=0}^{+\infty}\sum_{m=1}^{+\infty}\sum_{r=0}^{+\infty}B^{n,k}_{j,m,r}\frac{v^{\nu b^{n,k}_{j,m,r}-\beta^{n,k}_{j,m,r}}}{(\lambda+k\mu+v^\nu)^{b^{n,k}_{j,m,r}}}
\!\!-\!\!
\sum_{j=0}^{+\infty}\sum_{m=1}^{+\infty}\sum_{r=0}^{+\infty}C^{n,k}_{j,m,r}\frac{v^{\nu c^{n,k}_{j,m,r}-\gamma^{n,k}_{j,m,r}}}{(\lambda+k\mu+v^\nu)^{c^{n,k}_{j,m,r}}}.
\end{split}
\end{align*}
Finally, by using Eq. \eqref{eq:ML3LT} we obtain
\begin{align*}
\begin{split}
p_{n,k}^\nu(t)&=\sum_{j=0}^{+\infty}A^{n,k}_jt^{\alpha_j^{n,k}-1}E^{a^{n,k}_j}_{\nu,\alpha_j^{n,k}}(-(\lambda+k\mu)t^\nu)\\ &+\sum_{j=0}^{+\infty}\sum_{m=1}^{+\infty}\sum_{r=0}^{+\infty}B^{n,k}_{j,m,r}t^{\beta_{j,m,r}^{n,k}-1}E^{b^{n,k}_{j,m,r}}_{\nu,\beta_{j,m,r}^{n,k}}(-(\lambda+k\mu)t^\nu)\\ &-\sum_{j=0}^{+\infty}\sum_{m=1}^{+\infty}\sum_{r=0}^{+\infty}C^{n,k}_{j,m,r}t^{\gamma_{j,m,r}^{n,k}-1}E^{c^{n,k}_{j,m,r}}_{\nu,\gamma_{j,m,r}^{n,k}}(-(\lambda+k\mu)t^\nu).
\end{split}
\end{align*}
\end{proof}
%\begin{rmk}
%For $k=1$ let us observe that $s=k=1$. So we have
%\begin{align*}
%A^{n,1}_{j}&=\binom{n+2j}{n+j}\lambda^{n+j}\mu^j\\
%a^{n,1}_j&=n+2j+1\\
%\alpha^{n,1}_j&=\nu(n+2j)+1\\
%B^{n,1}_{j,m,r}&= \frac{m}{m+2r}\binom{m+2r}{m+r}\binom{n+2j}{n+j}\lambda^{n+j+r}\mu^{m+r+j}\\
%b^{n,1}_{j,m,r}&=m+2r+n+2j+1\\
%\beta^{n,1}_{j,m,r}&=\nu(m+2r+n+2j)+1\\
%C^{n,1}_{j,m,r}&=\frac{m}{m+2r}\binom{m+2r}{m+r}\binom{n+2j+1}{n+j+1}\lambda^{n+j+r+1}\mu^{m+r+j}\\
%c^{n,1}_{j,m,r}&=m+2r+n+2j+2\\
%\gamma^{n,1}_{j,m,r}&=\nu(m+2r+n+2j+1)+1
%\end{align*}
%and then
%\begin{align*}
%\begin{split}
%p_{n,1}^\nu(t)&=\sum_{j=0}^{+\infty}\binom{n+2j}{n+j}\lambda^{n+j}\mu^jt^{\nu(n+2j)}E^{n+2j+1}_{\nu,\nu(n+2j)+1}(-(\lambda+k\mu)t^\nu)\\ &+\sum_{j=0}^{+\infty}\sum_{m=1}^{+\infty}\sum_{r=0}^{+\infty}\frac{m}{m+2r}\binom{m+2r}{m+r}\binom{n+2j}{n+j}\lambda^{n+j+r}\mu^{m+r+j}\times \\ & \times t^{\nu(m+2r+n+2j)}E^{m+2r+n+2j+1}_{\nu,\nu(m+2r+n+2j)+1}(-(\lambda+k\mu)t^\nu)\\ &-\sum_{j=0}^{+\infty}\sum_{m=1}^{+\infty}\sum_{r=0}^{+\infty}\frac{m}{m+2r}\binom{m+2r}{m+r}\binom{n+2j+1}{n+j+1}\lambda^{n+j+r+1}\mu^{m+r+j}\times \\ & \times t^{\nu(m+2r+n+2j+1)}E^{m+2r+n+2j+2}_{\nu,\nu(m+2r+n+2j+1)+1}(-(\lambda+k\mu)t^\nu).
%\end{split}
%\end{align*}
%\end{rmk}
\section{Some features of the fractional $M/E_k/1$ queue}\label{sec6}
In this section we want to determine some features of the fractional $M/E_k/1$ queue. In particular we focus on the mean queue length (measured in phases), the distribution of the busy period and the distribution of the waiting times.
\subsection{Mean queue length}\label{subsect61}
To determine the mean queue length, let us consider the queue length process $\cL^\nu(t)$ and define:
\begin{equation*}
M^\nu(t)=\E[\cL^\nu(t)|\cL^\nu(0)=0].
\end{equation*}
Let us first show the following Theorem
\begin{thm}
The function $M^\nu(t)$ is solution of the following fractional Cauchy problem:
\begin{equation}\label{eq:meanCp}
\begin{cases}
D_t^\nu M^\nu=k(\lambda-\mu)+k\mu \cP^\nu_0(t) \\
M^\nu(0)=0.
\end{cases}
\end{equation}
\end{thm}
\begin{proof}
The initial condition follows obviously from the definition of $M^\nu(t)$.\\
To obtain the equation, recall that, by definition:
\begin{equation*}
M^\nu(t)=\sum_{n=1}^{+\infty}n\cP_n^\nu(t).
\end{equation*}
Thus, let us consider Eq. \eqref{eq:fracsysql}, multiply the second equation by $n$ and then sum over $n$ to obtain
\begin{equation}\label{eq:meaneqpass3}
D_t^\nu M^\nu(t)=-(\lambda+k\mu)M^\nu(t)+k\mu \sum_{n=1}^{+\infty}n\cP_{n+1}^\nu(t)+\lambda \sum_{n=1}^{+\infty}\sum_{m=1}^{n}c_m\cP^\nu_{n-m}(t).
\end{equation}
Let us observe that
\begin{equation}\label{eq:meaneqpass1}
\sum_{n=1}^{+\infty}n\cP_{n+1}^\nu(t)=M^\nu(t)-\sum_{n=0}^{+\infty}\cP_{n+1}^\nu(t)=M^\nu(t)+\cP_0^\nu(t)-1
\end{equation}
and
\begin{equation*}
\sum_{n=1}^{+\infty}\sum_{m=1}^{n}nc_m\cP_{n-m}^\nu(t)=\sum_{m=1}^{+\infty}\sum_{n=m}^{+\infty}nc_m\cP_{n-m}^\nu(t).
\end{equation*}
Now let us pose $r=n-m$ to obtain
\begin{align*}
\begin{split}
\sum_{m=1}^{+\infty}\sum_{n=m}^{+\infty}nc_m\cP_{m-n}^\nu(t)&=\sum_{m=1}^{+\infty}\sum_{r=0}^{+\infty}(r+m)c_m\cP_{r}^\nu(t)\\&=\sum_{m=1}^{+\infty}c_m\sum_{r=0}^{+\infty}r\cP_{r}^\nu(t)+\sum_{m=1}^{+\infty}mc_m\sum_{r=0}^{+\infty}\cP_{r}^\nu(t).
\end{split}
\end{align*}
Recalling that $c_m=\delta_{m,k}$, we have
\begin{equation}\label{eq:meaneqpass2}
\sum_{m=1}^{+\infty}c_m\sum_{r=0}^{+\infty}r\cP_{r}^\nu(t)+\sum_{m=1}^{+\infty}mc_m\sum_{r=0}^{+\infty}\cP_{r}^\nu(t)=M^\nu(t)+k.
\end{equation}
Using Eqs. \eqref{eq:meaneqpass1} and \eqref{eq:meaneqpass2} in \eqref{eq:meaneqpass3} we obtain the first equation of \eqref{eq:meanCp}.
\end{proof}
By integrating \eqref{eq:meanCp}, we obtain the following representation of the mean queue length:
\begin{equation}\label{eq:meanint}
M^\nu(t)=\frac{k(\lambda-\mu)t^\nu}{\Gamma(\nu+1)}+k\mu \cI^\nu_t \cP_0^\nu
\end{equation}
For $k=1$ we obtain the mean queue length of a fractional $M/M/1$ queue, as stated in Theorem $2.3$ of \cite{CahoyPolitoPhoha2015}. For $\nu=1$ we obtain Eq. \eqref{eq:meanql1}.\\
Starting from Eq. \eqref{eq:meanql1} we can also obtain a closed form for $M^\nu(t)$.
\begin{thm}
Let $M^\nu(t)=\E[\cL^\nu(t)|\cL^\nu(0)=0]$ and denote with $\overline{M}^\nu(v)$ its Laplace transform. Then
\begin{equation}\label{eq:mean}
M^\nu(t)=k(\lambda-\mu)\frac{t^\nu}{\Gamma(\nu+1)}+k\mu \sum_{m=1}^{+\infty}\sum_{r=0}^{+\infty}C^0_{m,r}t^{\gamma^M_{m,r}-1}E^{\delta^0_{m,r}}_{\nu,\gamma^M_{m,r}}(-(\lambda+k\mu)t^\nu)
\end{equation}
and
\begin{equation*}
\overline{M}^\nu(v)=k(\lambda-\mu)\frac{1}{v^{\nu+1}}+k\mu \sum_{m=1}^{+\infty}\sum_{r=0}^{+\infty}C^0_{m,r}\frac{v^{\nu \delta^0_{m,r}-\gamma^M_{m,r}}}{(\lambda+k\mu+v^\nu)^{\delta^0_{m,r}}}
\end{equation*}
where
\begin{equation}\label{eq:gammaM}
\gamma^M_{m,r}=\nu \delta^{0}_{m,r}+1
\end{equation}
and $C^0_{m,r}$ and $\delta^0_{m,r}$ are defined in \eqref{eq:C0mr}.
\end{thm}
\begin{proof}
From Eq. \eqref{eq:pnsnuint} and the equality $\cP_m^\nu=p_{n_k(m),s_k(m)}^\nu$ we easily obtain that
\begin{equation*}
M^\nu(t)=\int_0^{+\infty}M^1(y)\bP(L_\nu(t)\in dy).
\end{equation*}
Using \eqref{eq:meanql1} we obtain
\begin{equation*}
M^\nu(t)=k(\lambda-\mu)\int_0^{+\infty}y\bP(L_\nu(t)\in dy)+k\mu \int_0^{+\infty}\int_0^{y}\cP_0^1(z)dz\bP(L_\nu(t)\in dy).
\end{equation*}
Taking the Laplace transform we obtain
{\small\begin{equation}\label{eq:meanclforpass1}
\overline{M}^\nu(v)=k(\lambda-\mu)v^{\nu-1}\int_0^{+\infty}ye^{-yv^\nu}dy+k\mu v^{\nu-1}\int_0^{+\infty}\int_0^{y}\cP^1_0(z)e^{-yv^\nu}dzdy.
\end{equation}}
For the first integral we easily obtain
\begin{equation}\label{eq:meanclforpass2}
\int_0^{+\infty}ye^{-yv^\nu}dy=\frac{1}{v^{2\nu}}.
\end{equation}
For the second one we have
\begin{align}\label{eq:meanclforpass3}
\begin{split}
\int_0^{+\infty}\int_0^{y}\cP^1_0(z)e^{-yv^\nu}dzdy&=\int_0^{+\infty}\int_z^{+\infty}\cP^1_0(z)e^{-yv^\nu}dydz\\&=\frac{1}{v^\nu}\int_0^{+\infty}\cP_0^1(z)e^{-zv^\nu}dz.
\end{split}
\end{align}
By using Eq. \eqref{eq:p01} we obtain
\begin{align}\label{eq:meanclforpass4}
\begin{split}
\int_0^{+\infty}\cP_0^1(z)e^{-zv^\nu}dz&=\sum_{m=1}^{+\infty}\sum_{r=0}^{+\infty}\frac{m\lambda^r(k\mu)^{m+rk-1}}{r!\Gamma(n+rk+1)}
\int_0^{+\infty}\hspace*{-0.5cm}z^{m+r(k+1)-1}e^{-(\lambda+k\mu+v^\nu)z}dz\\&=\sum_{m=1}^{+\infty}\sum_{r=0}^{+\infty}C^0_{m,r}\frac{1}{(\lambda+k\mu+v^\nu)^{m+r(k+1)}}.
\end{split}
\end{align}
Using Eqs. \eqref{eq:meanclforpass4}, \eqref{eq:meanclforpass3} and \eqref{eq:meanclforpass2} in \eqref{eq:meanclforpass1} we obtain
\begin{equation*}
\overline{M}^\nu(v)=k(\lambda-\mu)\frac{1}{v^{\nu+1}}+k\mu \sum_{m=1}^{+\infty}\sum_{r=0}^{+\infty}C^0_{m,r}\frac{v^{-1}}{(\lambda+k\mu+v^\nu)^{m+r(k+1)}}.
\end{equation*}
Posing then $\gamma^M_{m,r}$ as in \eqref{eq:gammaM} we obtain
\begin{equation*}
\overline{M}^\nu(v)=k(\lambda-\mu)\frac{1}{v^{\nu+1}}+k\mu \sum_{m=1}^{+\infty}\sum_{r=0}^{+\infty}C^0_{m,r}\frac{v^{\nu \delta^0_{m,r}-\gamma^M_{m,r}}}{(\lambda+k\mu+v^\nu)^{\delta^0_{m,r}}}.
\end{equation*}
Finally, taking the inverse Laplace transform and recalling eq. \eqref{eq:ML3LT} we obtain
\begin{equation*}
M^\nu(t)=k(\lambda-\mu)\frac{t^\nu}{\Gamma(\nu+1)}+k\mu \sum_{m=1}^{+\infty}\sum_{r=0}^{+\infty}C^0_{m,r}t^{\gamma^M_{m,r}-1}E^{\delta^0_{m,r}}_{\nu,\gamma^M_{m,r}}(-(\lambda+k\mu)t^\nu).
\end{equation*}
\end{proof}
Comparing Eq. \eqref{eq:meanint} with \eqref{eq:mean} we also obtain some information on the fractional integral of $\cP_0^\nu(t)$. In particular
\begin{equation*}
\cI^\nu_t \cP_0^\nu=\sum_{m=1}^{+\infty}\sum_{r=0}^{+\infty}C^0_{m,r}t^{\gamma^M_{m,r}-1}E^{\delta^0_{m,r}}_{\nu,\gamma^M_{m,r}}(-(\lambda+k\mu)t^\nu).
\end{equation*}
\subsection{Distribution of the busy period}\label{subsec62}
We can obtain the distribution of the busy period for a fractional $M/E_k/1$ queue by following the lines of \cite{Luchak1956}.
\begin{thm}
Let $B^\nu$ be the duration of the busy period of a fractional $M/E_k/1$ queue, $B^\nu(t)$ be its distribution function and $\overline{B}^\nu(v)$ be its Laplace transform. Then we have
\begin{equation*}
B^\nu(t)=\sum_{r=0}^{+\infty}C^B_rt^{\gamma^B_r-1}E^{\delta^B_r}_{\nu,\gamma^B_r}(-(\lambda+k\mu)t^\nu)
\end{equation*}
and
\begin{equation*}
\overline{B}^\nu(v)=\sum_{r=0}^{+\infty}C^B_r\frac{v^{\nu\delta^B_r-\gamma^B_r}}{(\lambda+k\mu+v^\nu)^{\delta^B_r}}
\end{equation*}
where
\begin{align}\label{eq:CBgammaB}
C^B_r=k\mu C^0_{k,r} && \gamma^B_r=\nu \delta^0_{k,r}+1
\end{align}
with $C^0_{k,r}$ and $\delta^0_{k,r}$ defined in Eqs. \eqref{eq:C0mr}. 
\end{thm}
\begin{proof}
Let us consider the process $\widetilde{\cL}^1(t)$ whose state probabilities
\begin{equation*}
\widetilde{\cP}_n^1(t)=\bP(\widetilde{\cL}^1(t)=n|\widetilde{\cL}^1(0)=1)
\end{equation*}
satisfy the following Cauchy problem:
\begin{equation*}
\begin{cases}
\der{\wcP_0}{t}(t)=k\mu \wcP_1(t) \\
\der{\wcP_1}{t}(t)=-(\lambda+k\mu) \wcP_1(t)+k\mu \wcP_2(t) \\
\der{\wcP_n}{t}(t)=-(\lambda+k\mu) \wcP_{n}(t)+k\mu \wcP_{n+1}(t)+\lambda \sum_{m=1}^{n}c_m\wcP_{n-m}(t) & n \ge 2\\
\wcP_n(0)=\delta_{n,k} & n \ge 0.
\end{cases}
\end{equation*}
The process $\wcL^1(t)$behaves as $\cL^1(t)$ except that it starts from $k$ (that is to say when the first customer entered the queue) instead that from $0$ and admits $0$ as absorbent state. In particular, by construction, $\wcP_0^1(t)=B^1(t)$.\\
Now let us consider $\wcL^\nu(t):=\wcL^1(L_\nu(t))$. As done with the interarrival and the inter-phase times, if $T_n$ is the $n$-th jump time of the process $\cL^\nu(t)$, then $\cL^\nu(T_n)$ is a time-homogeneous Markov chain, then, by construction, $\wcP_0^\nu(t)=B^\nu(t)$.\\
From the relation between $\wcL^1$ and $\wcL^\nu$, we have that
\begin{equation*}
B^\nu(t)=\int_0^{+\infty}B^1(y)\bP(L_\nu(t)\in dy).
\end{equation*}
Thus, by using Eq. \eqref{eq:B1} we have
\begin{equation*}
B^\nu(t)=\sum_{r=0}^{+\infty}\frac{k\lambda^r(k\mu)^{k(r+1)}}{r!\Gamma(rk+k+1)}\int_0^{+\infty}\int_0^{y}z^{k+r(k+1)-1}e^{-(\lambda+k\mu)z}dz\bP(L_\nu(t)\in dy).
\end{equation*}
Taking the Laplace transform we have
\begin{equation}\label{eq:Bpass1}
\overline{B}^\nu(v)=\sum_{r=0}^{+\infty}\frac{k\lambda^r(k\mu)^{k(r+1)}}{r!\Gamma(rk+k+1)}v^{\nu-1}\int_0^{+\infty}\int_0^{y}z^{k+r(k+1)-1}e^{-(\lambda+k\mu)z}e^{-yv^\nu}dzdy.
\end{equation}
Now let us consider the integral and observe that:
\begin{align}\label{eq:Bpass2}
\begin{split}
\int_0^{+\infty}\int_0^{y}z^{k+r(k+1)-1}&e^{-(\lambda+k\mu)z}e^{-yv^\nu}dzdy\\
&\hspace*{-2cm}=\int_0^{+\infty}\int_z^{+\infty}z^{k+r(k+1)-1}e^{-(\lambda+k\mu)z}e^{-yv^\nu}dydz\\
&\hspace*{-2cm}=\frac{1}{v^\nu}\int_0^{+\infty}z^{k+r(k+1)-1}e^{-(\lambda+k\mu+v^\nu)z}dz
=\frac{1}{v^\nu}\frac{(k+r(k+1)-1)!}{(\lambda+k\mu+v^\nu)^{k+r(k+1)}}.
\end{split}
\end{align}
Using Eq. \eqref{eq:Bpass2} in Eq. \eqref{eq:Bpass1} we obtain
\begin{equation*}
\overline{B}^\nu(v)=\sum_{r=0}^{+\infty}\frac{k\lambda^r(k\mu)^{k(r+1)}(k+r(k+1)-1)!}{r!\Gamma(rk+k+1)}\frac{v^{-1}}{(\lambda+k\mu+v^\nu)^{k+r(k+1)}}.
\end{equation*}
Thus, posing $C^B_r$ and $\gamma^B_r$ as in \eqref{eq:CBgammaB} we obtain
\begin{equation*}
\overline{B}^\nu(v)=\sum_{r=0}^{+\infty}C^B_r\frac{v^{\nu\delta^B_r-\gamma^B_r}}{(\lambda+k\mu+v^\nu)^{\delta^B_r}}.
\end{equation*}
Finally, taking the inverse Laplace transform and using Eq. \eqref{eq:ML3LT}, we obtain
\begin{equation*}
B^\nu(t)=\sum_{r=0}^{+\infty}C^B_rt^{\gamma^B_r-1}E^{\delta^B_r}_{\nu,\gamma^B_r}(-(\lambda+k\mu)t^\nu).
\end{equation*}
\end{proof}
\subsection{Distribution of some conditioned waiting times}\label{subsec63}
For classical single-channel queues, a method to obtain the waiting time distribution has been presented in \cite{Gaver1954}. However, this technique makes an explicit use of the Markov property of a classical single-channel queue. Our fractional queues are semi-Markov processes, then this technique cannot be used in our context. However, to give some information on the waiting times we can introduce some form of conditioning.\\
Before working with the waiting times, let us introduce a new distribution function that will be useful in the following. Consider a random variable $T \sim ML_\nu(\lambda)$ and fix $t_0>0$. Observe that
\begin{equation*}
\bP(T\le t_0+t|T \ge t_0)=1-\frac{E_\nu(-\lambda(t_0+t)^\nu)}{E_\nu(-\lambda t_0^\nu)}.
\end{equation*}
The explicit dependence of this probability with respect to $t_0$ is due to the lack of semigroup property that is typical of Mittag-Leffler functions (see, for instance, \cite{KexuePeng2010}). This observation leads us to the definition of a new distribution.\\
Consider a random variable $T$ such that its distribution function is given by
\begin{equation*}
F_T(t)=1-\frac{E_\nu(-\lambda(t_0+t)^\nu)}{E_\nu(-\lambda t_0^\nu)};
\end{equation*}
we call such random variable \textit{residual Mittag-Leffler} random variable starting from $t_0$ with fractional index $\nu$ and parameter $\lambda$ and we denote it by $T \sim RML_\nu(t_0,\lambda)$. Denote
\begin{equation*}
\Phi(t):=E_\nu(-\lambda(t_0+t)^\nu)
\end{equation*}
and observe that, by the definition of Mittag-Leffler function given in \eqref{eq:ML1}, we have
\begin{equation*}
\Phi(t)=\sum_{k=0}^{+\infty}\frac{(-\lambda)^k}{\Gamma(\nu k+1)}(t_0+t)^{\nu k}.
\end{equation*}
Denote with $\overline{\Phi}(v)$ its Laplace transform and observe that
\begin{equation*}
\overline{\Phi}(v)=\sum_{k=0}^{+\infty}\frac{(-\lambda)^k}{\Gamma(\nu k+1)}e^{vt_0}v^{-1-\nu k}\Gamma(\nu k+1,vt_0), \ v>0,
\end{equation*}
where $\Gamma(z,s)$ is the analytic extension of the upper incomplete gamma function on $\C^2$ (see, for instance, \cite{Lang2013,Olver1997}). Denote with $f_T(t)$ the probability density function of $T$ and observe that
\begin{equation*}
f_T(t)=-\frac{1}{E_\nu(-\lambda t_0^\nu)}\der{\Phi(t)}{t}.
\end{equation*}
Denoting with $\overline{f}_T(v)$ its Laplace transform, we obtain
\begin{equation*}
\overline{f}_T(v)=1-\frac{e^{vt_0}}{E_\nu(-\lambda t_0^\nu)}\sum_{k=0}^{+\infty}\frac{(-\lambda)^k}{\Gamma(\nu k+1)}v^{-\nu k}\Gamma(\nu k+1,vt_0), \ v>0.
\end{equation*}
Finally let us also observe that if $T \sim RML_\nu(t_0,\lambda)$ for $\nu=1$, then $T \sim Exp(\lambda)$. Now we are ready to determine such waiting times.\\
Let $W^\nu_t$ be the waiting time of a customer that enters the queue at time $t>0$. Let us define the random set
\begin{equation*}
\cA_t^\nu(\omega)=\{0 \le s \le t: \ \cL^\nu(s-)(\omega)=\cL^\nu(s)(\omega)+1\}
\end{equation*}
that is the set of the time instants that precede $t$ in which a phase has been completed. Let us define also the random variable
\begin{equation*}
\cN_t^\nu=|\cA^\nu_t|
\end{equation*}
which is the number of phase terminated before $t$. Finally, let us define the random variable
\begin{equation*}
\Theta_t^\nu=\begin{cases} \sup \cA^\nu_t & \cN_t^\nu>0\\
t & \cN_t^\nu=0
\end{cases}
\end{equation*}
that is the last instant before $t$ in which a phase has been completed. Let us define:
\begin{equation*}
w^\nu(\xi;t,t_0,n)d\xi=\bP(W^\nu_t \in d\xi| \Theta_t^\nu=t_0, \cL^\nu(t)=n).
\end{equation*}
Since we have conditioned this probability with $\cL^\nu(t)=n$, we know that a customer has to wait $n$ independent phases to be terminated. Thus we have
\begin{equation*}
\E[W^\nu_t|\Theta_t^\nu=t_0, \cL^\nu(t)=n]=\sum_{j=1}^{n}T_j
\end{equation*}
where $T_j$ are $n$ independent random variables. For $1 \le j \le n-1$, these are just Mittag-Leffler random variables with $T_j \sim ML_\nu(k\mu)$, thus, denoting $\widetilde{W}=\sum_{j=1}^{n-1}T_j$, we know that $\widetilde{W}\sim GE_\nu(n-1,k\mu)$. For $T_n$, we know that the last time a phase has been completed is $t_0$, since we have conditioned the probability with $\Theta_t^\nu=t_0$. So we know that a customer is in the current phase for $t-t_0$. For such reason, we know that $T_n \sim RML_\nu(t-t_0,k\mu)$. Thus we have that
\begin{equation*}
w^\nu(\xi;t,t_0,n)=f_{\widetilde{W}}\ast f_{T_n}(\xi).
\end{equation*}
In particular we can obtain the Laplace transform $\overline{w}^\nu(v;t,t_0,n)$ of $w^\nu(\xi;t,t_0,n)$ as
\begin{multline*}
\overline{w}^\nu(v;t,t_0,n)=\frac{(k\mu)^{n-1}}{(k\mu+v^\nu)^{n-1}}\times \\ \times\left[1-\frac{e^{v(t-t_0)}}{E_\nu(-\lambda (t-t_0)^\nu)}\sum_{k=0}^{+\infty}\frac{(-k\mu)^k}{\Gamma(\nu k+1)}v^{-\nu k}\Gamma(\nu k+1,v(t-t_0))\right].
\end{multline*}
For $\nu=1$, we know that $T_n \sim Exp(k\mu)$ and $\widetilde{W}\sim E_{n-1}(k\mu)$, thus $T_n+\widetilde{W}\sim E_{n}(k\mu)$ and we know that
\begin{equation*}
w^1(\xi;t,t_0,n)=\frac{(k\mu)^{n}\xi^{n-1}e^{-k\mu \xi}}{(n-1)!}=w^1(\xi;t,n)
\end{equation*}
since the density is independent from $t_0$. Thus we obtain
\begin{equation*}
w^1(\xi;t)=\sum_{n=0}^{+\infty}\cP_n^1(t)w^1(\xi;t,n)
\end{equation*}
that is Eq. \eqref{eq:w1unco}, that is the one given in \cite{Gaver1954}.
\section{Sample paths simulation}\label{sec7}
In this section we want to give some algorithm to simulate such queue process. To do this, we need to show how to simulate a Mittag-Leffler random variable. Let us show two preliminary simple lemmas.
\begin{lem}
Let $T \sim Exp(\lambda)$ and $\sigma_\nu(t)$ be a $\nu$-stable subordinator independent of $T$. Then $\sigma_\nu(T)\sim ML_\nu(\lambda)$
\end{lem}
\begin{proof}
Let us just observe that
\begin{align*}
\bP(\sigma_\nu(T)>t)&=\bP(T>L_\nu(t))\\&=\int_0^{+\infty}\bP(T>y)\bP(L_\nu(t)\in dy)\\&=\int_0^{+\infty}e^{-\lambda y}\bP(L_\nu(t)\in dy)\\&=E_\nu(-\lambda t^\nu)
\end{align*}
where the last equality is given by the Laplace transform of the density $\bP(L_\nu(t)\in dy)$ of the inverse $\nu$-stable subordinator with respect to $y$, given in \cite{Bingham1971}.
\end{proof}
\begin{lem}
Let $\sigma_\nu(t)$ be a $\nu$-stable subordinator and $T$ be a non-negative random variable independent of $\sigma_\nu$. Then
\begin{equation*}
\sigma_\nu(T)\overset{d}{=}T^{\frac{1}{\nu}}\sigma_\nu(1)
\end{equation*}
\end{lem}
\begin{proof}
Let us just observe that
\begin{align*}
\bP(\sigma_\nu(T)>t)&=\int_{0}^{+\infty}\bP(\sigma_\nu(y)>t)\bP(T \in dy)\\&=\int_0^{+\infty}\bP(y^{\frac{1}{\nu}}\sigma_\nu(1)>t)\bP(T \in dy)\\&=\bP(T^{\frac{1}{\nu}}\sigma_\nu(1))
\end{align*}
where the second equality is given by $\sigma_\nu(t)\overset{d}{=}t^{\frac{1}{\nu}}\sigma_\nu(1)$.
\end{proof}
By using these two lemmas it is easy to show that
\begin{cor}
Let $T \sim ML_\nu(\lambda)$, $\sigma_\nu(t)$ be a $\nu$-stable subordinator and $S \sim Exp(\lambda)$ independent from $\sigma_\nu(t)$. Then $T \overset{d}{=}S^{\frac{1}{\nu}}\sigma_\nu(1)$.
\end{cor}
To simulate the behaviour of the queue process $Q^\nu(t)$, we will use a modified version of the Gillespie algorithm (see \cite{Gillespie1977}). This modified version, already introduced in \cite{CahoyPolitoPhoha2015}, is based on the fact that if $T_n$ are the jump times of $Q^\nu(t)$, then $Q^\nu(T_n)$ is still a Markov chain with the same transition probabilities as $Q^1(T_n)$. Thus we need first to simulate the jump times $T_n$ and then to define what event will happen in such time. In particular we will follow these steps
\begin{enumerate}
\item[Step 1] Fix $Q^\nu(0)$ and $T_0=0$;
\item[Step 2] Suppose have simulated the queue up to the $n$-th event that happens in time $T_n$. If $Q(T_n)=(0,0)$ simulate a random variable $I \sim ML_\nu(\lambda)$, if $Q(T_n) \in \cS^*$ simulate a random variable $I \sim ML_\nu(\lambda+k\mu)$;
\item[Step 3] Obtain the random variable $T_{n+1}=T_n+I$;
\item[Step 4] If $Q^\nu(T_{n})=(0,0)$, pose $Q^\nu(T_{n+1})=(1,k)$. If $Q^\nu(T_{n}) \in \cS^*$, then simulate a uniform variable $U$ in $(0,1)$. If $U<\frac{\lambda}{\lambda+k\mu}$, then pose $N(T_{n+1})=N(T_n)+1$ and $S(T_{n+1})=S(T_n)$. If $U\ge \frac{\lambda}{\lambda+k\mu}$, distinguish three cases:
\begin{itemize} 
\item if $S(T_n)>1$, then pose $N(T_{n+1})=N(T_n)$ and $S(T_{n+1})=S(T_n)-1$;
\item if $S(T_n)=1$ and $N(T_n)>1$, then pose $S(T_{n+1})=k$ and $N(T_{n+1})=N(T_n)-1$;
\item if $S(T_n)=1$ and $N(T_n))=1$, then pose $S(T_{n+1})=N(T_{n+1})=0$;
\end{itemize}
\item[Step 5] To obtain the whole process $Q^\nu(t)$, one has just to pose $Q^\nu(t)=Q^\nu(T_{n-1})$ for any $t \in [T_{n-1},T_n)$.
\end{enumerate}
In Figure \ref{fig:sperlang} we provide the simulation of a sample path of the process $N^\nu(t)$ for a fractional Erlang queue and we show its embedded Markov chain.
\begin{figure}[h]
\centering
\includegraphics[width=0.49\linewidth]{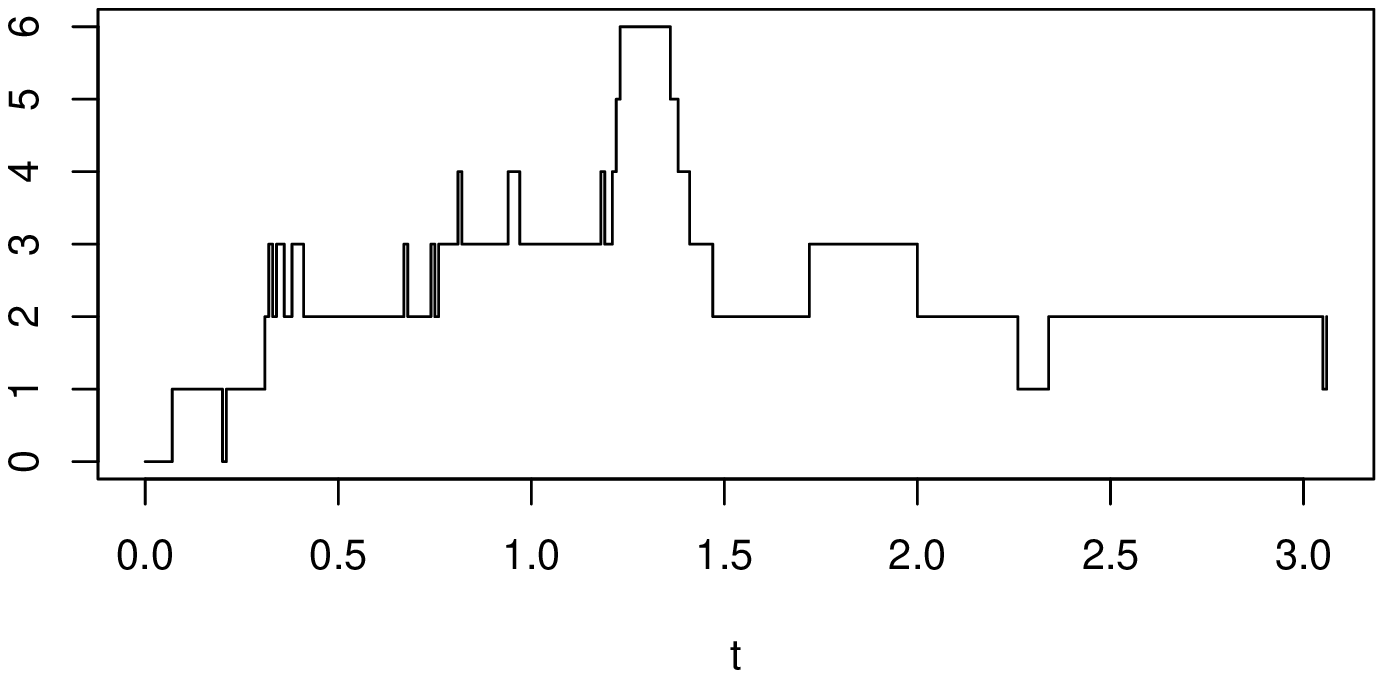}
\includegraphics[width=0.49\linewidth]{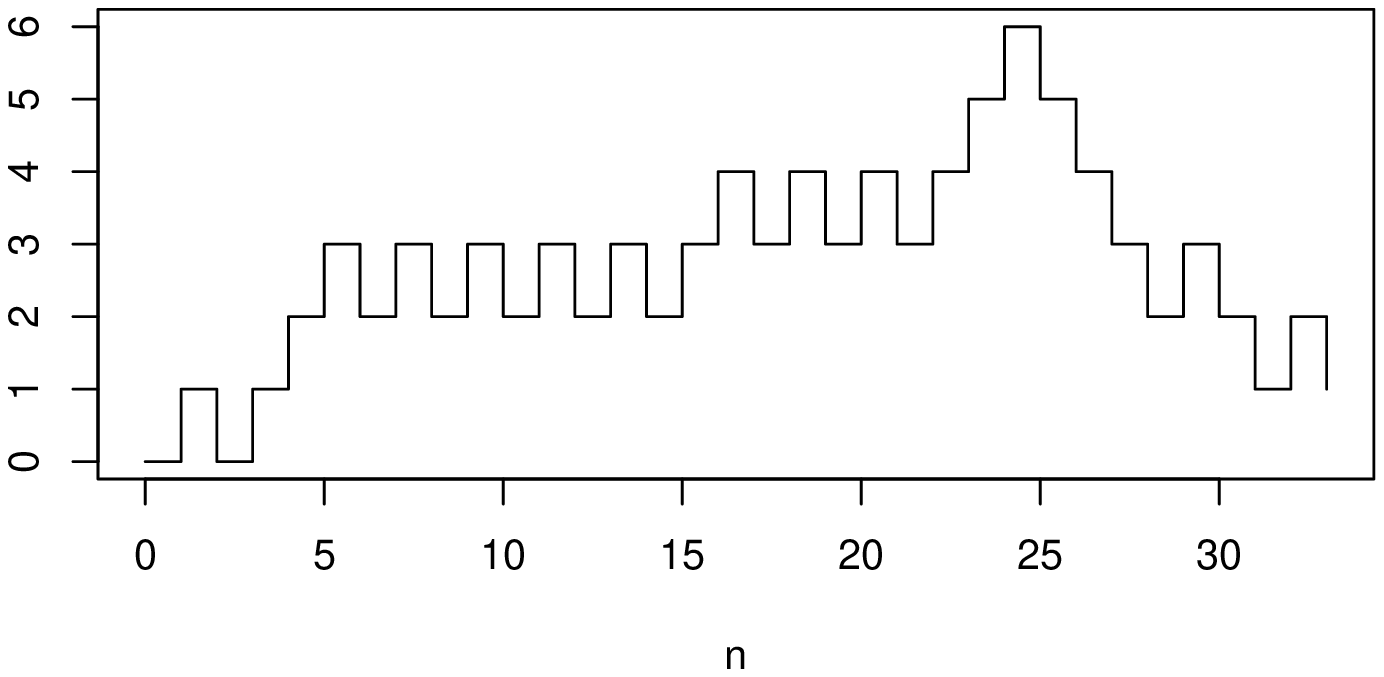}
\caption{Sample paths simulation. On the left, a sample path of $N^\nu(t)$. On the right, its embedded Markov chain $N^\nu(J_n)$ where $J_n$ are the jump times of $N^\nu(t)$. We set $k=2$, $\lambda=4$, $\mu=5$ and $\nu=0.75$.}
\label{fig:sperlang}
\end{figure}
Another way one can use to simulate such process, is to first simulate a classical Erlang queue $Q(t)$, then to simulate an inverse $\nu$-stable subordinator $L_\nu(t)$ (some numerical approach to inverse subordinators are described in \cite{VeilletteTaqquMurad2010a,VeilletteTaqquMurad2010b}, while some simulation algorithms for the stable subordinator are given in \cite{AsmussenGlynn2007}) and finally to combine the two simulated process in $Q^\nu(t):=Q(L_\nu(t))$ (see, for instance, \cite[Example $5.21$]{MeerschaertSikorskii2011}). In this way, the fractional Erlang queue is obtained as a dilatation in time of the classical one, thus we can explicitly see the effect of the inverse $\nu$-stable subordinator, as it is done in Figure \ref{fig:erlangnontc}. 
\begin{figure}[h]
\centering
\includegraphics[width=0.49\linewidth]{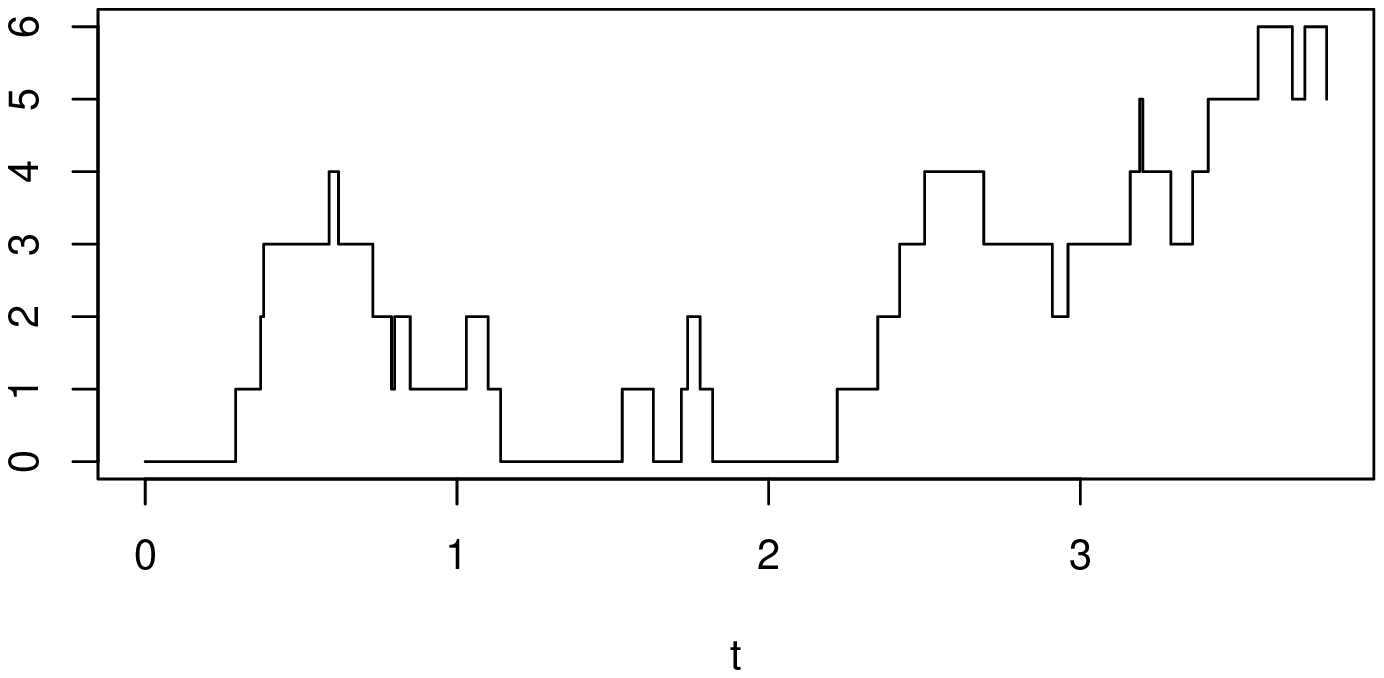}
\includegraphics[width=0.49\linewidth]{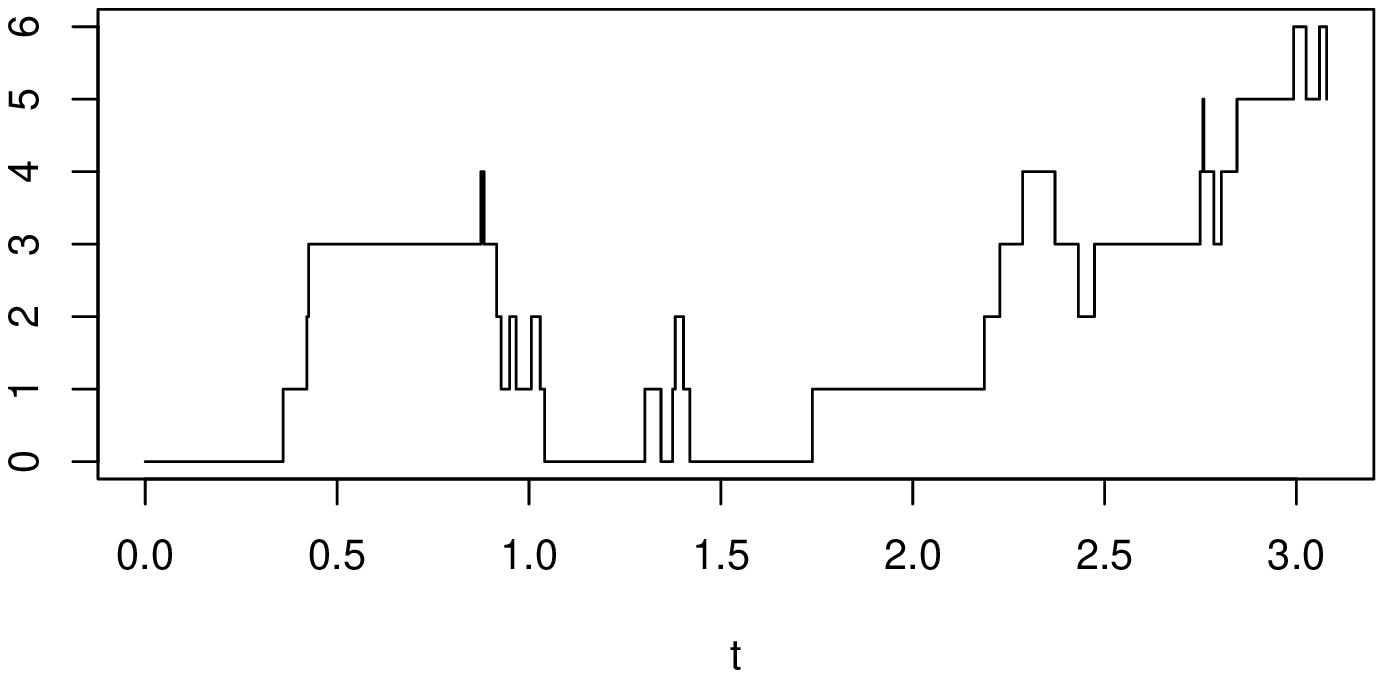}
\caption{Sample paths simulation. On the left, a sample path of $N(t)$ for a classical $M/E_k/1$ queue. On the right, a sample path of $N(L_\nu(t))$. We set $k=2$, $\lambda=4$, $\mu=5$ and $\nu=0.75$.}
\label{fig:erlangnontc}
\end{figure}
\section*{Acknowledgments}
N. Leonenko was supported in particular by Australian Research Council's Discovery Projects funding scheme (project DP160101366), and by project MTM2015-71839-P of MINECO, Spain (co-funded with FEDER funds).\\
E. Pirozzi was supported by INDAM-GNCS and MANM.
\bibliographystyle{plain}
\bibliography{biblio}

\end{document}